\documentclass[final,reqno,twoside,a4paper,11pt]{amsart}
\usepackage{mltools}
\usepackage{math62,mlthm10-REAR}

\usepackage{upref,amsmath,amssymb,amsthm,mathrsfs}
\usepackage{tikz, graphics, latexsym }
\usepackage[colorlinks=true,linkcolor=blue,citecolor=blue]{hyperref}

\usepackage[scaled]{helvet} 
\usepackage{courier} 
\usepackage[mathbf]{euler}
\usepackage{mllocal}

\numberwithin{equation}{section}
\begin{document}
\title[Sturm--Liouville operators with quadratic potentials at infinity]
{Zeta-determinants of Sturm--Liouville operators with quadratic potentials at infinity}

\author{Luiz Hartmann}
\address{Department of Mathematics, 
	Federal University of S\~ao Carlos (UFSCar),
	Brazil}
\email{hartmann@dm.ufscar.br}
\urladdr{http://www.dm.ufscar.br/profs/hartmann}

\author{Matthias Lesch}
\address{Mathematisches Institut,
Universit\"at Bonn,
Endenicher Allee 60,
53115 Bonn,
Germany}
\email{ml@matthiaslesch.de, lesch@math.uni-bonn.de}
\urladdr{www.matthiaslesch.de, www.math.uni-bonn.de/people/lesch}
\thanks{Partially supported by the 
        Hausdorff Center for Mathematics, Bonn}

\author{Boris Vertman} 
\address{Institute of Mathematics and Computer Science, University of M\"unster, Germany} 
\email{vertman@uni-muenster.de}
\urladdr{http://wwwmath.uni-muenster.de/42/arbeitsgruppen/ag-differentialgeometrie/}

\subjclass[2010]{58J52; 34B24}

\begin{abstract}
We consider Sturm--Liouville operators on a half line
$[a,\infty), a>0$, with potentials that are growing at most
quadratically at infinity. Such operators arise naturally in the
analysis of hyperbolic manifolds, or more generally manifolds with
cusps. We establish existence and a formula for the associated
zeta-determinant in terms of the Wronski-determinant of a
fundamental system of solutions adapted to the boundary
conditions.  Despite being the natural objects in the context of
hyperbolic geometry, spectral geometry of such operators has only
recently been studied in the context of analytic torsion.
\end{abstract}

\maketitle

\tableofcontents
\section{Introduction and formulation of the main results}
\label{s.intro}

In this paper we will investigate the zeta-determinant of Sturm--Liouville
operators with potentials that are growing quadratically at infinity. More
precisely, we consider operators of the form
\begin{equation}
  H  = -\frac{d}{dx}\Bigl( x^2 \frac{d}{dx}\cdot \Bigr) + x^2\mu^2 -\frac 14  +  V(x)
     =: D_\mu + V(x) 
     \label{eq.INTRO.cusp.op}
\end{equation}
on the interval $[a,\infty)$, $a>0$, with $\mu >0$ and only minimal regularity 
assumptions on 
the
potential $V$.  Ignoring the potential $V$ for a moment, such operators are
also referred to as totally characteristic operators and have been studied by
Melrose and Mendoza in \cite{MM}. However, the relation to our analysis here is
only formal, since \cite{MM} studies operators of totally characteristic
type near $x=0$, while here we are interested in the behavior of such
operators as $x$ approaches infinity.

Our motivation for looking at zeta-determinants of such operators arises from
geometry of hyperbolic manifolds or more generally manifolds with cusps.
Spectral geometry of such manifolds has been initiated by M\"uller in his
paper \cite{Mue-cusp}.  A recent work by the third named author \cite{Ver}
discusses analytic torsion on such spaces and in particular strongly relies on
computations of zeta-determinants of such operators.

With the present paper we intend to initiate further discussion of such
operators, parallel to developments in the setting of regular singular
operators, which in turn are motivated by the geometry of spaces with isolated
conical singularities.  Analysis of such spaces has been initiated by Cheeger
in his seminal papers \cite{Che:SGS}, \cite{Che:SGSR}, and corresponding 
zeta-determinants have been considered by the second named author in
\cite{Les:DRS}.  These results have been employed in various studies of
analytic torsion on conical singularities by the first named author jointly
with Spreafico \cite{HS}, \cite{HS-R-torsion} as well the third author jointly
with M\"uller \cite{MV}. We expect a similar impact of our discussion here in
the setting of manifolds with cusps. 

Before stating the main result of our paper, let us briefly recall the formula
for the $\zeta$--determinant of a second order Sturm--Liouville operator on a
finite interval with separated boundary conditions, \cf \cite{BFK:DEB}. Let
\begin{equation} \label{INTRO.eq.1}
  H_0 = -\frac{d}{dx}\Bigl( p(x) \frac{d}{dx}\cdot \Bigr) + V(x)
\end{equation}
be a differential operator on the finite interval $[a,b]$. Here, $p,
V$ are smooth functions and $p(x) > 0$ for all $x\in[a,b]$. We impose
separated boundary conditions at $a,b$ of the form
\begin{equation} \label{INTRO.eq.2}
  R_c f := \sin\theta_c \cdot f'(c) + \cos\theta_c \cdot f(c),\quad
   0\le \theta_c<\pi, c\in\{a,b\}.
\end{equation}  
A solution of the homogeneous equation $H_0 g=0$ is called
\emph{normalized} at $c$ if $R_c g  = 0$ and (we set $\operatorname{sgn} (a) = 1,  
\operatorname{sgn}(b) = -1$)
\begin{equation}\label{INTRO.eq.3}
\begin{split}
  g'(c) & = \operatorname{sgn}(c)\cdot p(c)^{-3/4}, \quad
  \text{ if } \theta_c = 0 \text{ (Dirichlet) },\\
  g(c)  & = p(c)^{-1/4},  \quad \quad \quad \text{ if } \theta_c > 0 \text{ (generalized
Neumann) }.
\end{split}
\end{equation}  
One might wonder where this normalization comes from. For a regular operator of the form 
\Eqref{INTRO.eq.1} there is a coordinate transformation $y(x):=\int_a^x 
p(x')^{-\frac{1}{2}} dx'$, which unitarily transforms the operator into a 
Sturm--Liouville operator of the form $-\partial_y^2 + V$. The known normalization for 
the latter operator, \cf \cite{Les:DRS}, is equivalent to \Eqref{INTRO.eq.3} under the 
transformation.

With this notation the following Theorem, which is a special case of a more
general result due to Burghelea, Friedlander and Kappeler, holds.

\begin{theorem}[\cite{BFK:DEB}]\label{T.BFK}
Let $\varphi,\psi$ be a fundamental system
of solutions to the differential equation $H_0 g=0$ with $R_a \varphi
=0, R_b\psi = 0$ and $\varphi$, $\psi$ being both normalized in the sense of
\Eqref{INTRO.eq.3}. Then the realization $H_0=H_0(R_a,R_b)$ of $H_0$ with respect
to the boundary conditions $R_a, R_b$ is self-adjoint and discrete.
Its $\zeta$--function has a meromorphic continuation to the complex plane
with simple poles. $0\in \C$ is not a pole and moreover the
$\zeta$--regularized determinant is given by
\begin{equation}\label{INTRO.eq.4}
  \detz(H_0(R_a,R_b)) = 2\cdot p\cdot W(\psi,\varphi) = 2\cdot p\cdot (\psi\cdot
   \varphi'-\psi'\cdot\varphi).
\end{equation}
Note that the Wronskian $p\cdot W(\psi,\varphi)$ is constant.
\end{theorem}   

This result has been generalized to regular singular operators by the
second and third named authors, \cf \cite{Les:DRS, LesVer}. 

In this paper we prove the analogue of BFK's Theorem for operators of
the form \Eqref{eq.INTRO.cusp.op}. To the best of our knowledge this is the
first example of a singular operator on an unbounded interval for
which such a formula for the $\zeta$--determinant is proven. For a function $f$ we will
use the corresponding capital letter $F$ to denote the multiplication operator by $f$.
\textit{E.~g.}, the multiplication operator by the coordinate function $x$ is denoted by $X$.

\begin{theorem}\label{T.main}
Fix any $\nu\geq0$ and suppose the potential $V$ in the differential expression
\Eqref{eq.INTRO.cusp.op} satisfies $V \in X^\gamma  L^1[a,\infty)$ for a fixed
$\gamma < 2$. We impose boundary conditions at $x=a$ of the form \Eqref{INTRO.eq.2}. 
The operator $H$ is in the limit point case at infinity and hence essentially
self-adjoint on the core domain
$\bigsetdef{ f\in C^\infty_0[a,\infty)}{ R_af = 0}$.
By abuse of notation we denote by $H = H(R_a)$ this self-adjoint realization.  Choose a fundamental system of 
solutions $\phi, \psi$ to $(H + \nu^2 )f=0$, 
where $R_a \phi = 0$ satisfies the boundary conditions at the left
end point and $\psi \in L^2 [a,\infty) $ is square integrable.
We normalize $\phi$ as above in \Eqref{INTRO.eq.3} and $\psi$ by
\begin{equation}
\lim\limits_{x\to \infty}  \psi(x) \cdot \sqrt{x}  \cdot K_\nu(\mu x)^{-1} = 1.
\end{equation}
Here $K_\nu$ is the modified Bessel function of the second kind of order $\nu$. 

Then $H(R_a)+\nu^2$ is self-adjoint with a discrete spectrum. Furthermore, its
$\zeta$--function admits a meromorphic continuation into a half plane
$\bigsetdef{z\in\C}{\Re z > r}$ for some $r<0$ with $0$ being a
regular point. Therefore, its zeta-regularized determinant is
well--defined. 

Furthermore, we have the explicit formula
\begin{equation}
  \detz (H(R_a)+\nu^2) = \sqrt{\frac 2\pi } \cdot a^2\cdot W(\psi, \phi)(a).
\end{equation}
\end{theorem}
Note that when comparing with \Eqref{INTRO.eq.4} we have $p(x) = x^2$. 

\subsection{Outline of proof and further results}

\subsubsection{The model operator}
The operator $H$ is treated as a perturbation of the model
\emph{cusp} operator parametrized by $\mu$
\begin{equation}\label{eq.INTRO.model}
D_\mu = -\frac{d}{dx}\Bigl( x^2 \frac{d}{dx}\cdot \Bigr) + x^2\mu^2 - \frac{1}{4}.
\end{equation}
We observe that if $\mu = 0$ then $D_0$ has a continuous spectrum, therefore we consider 
$\mu>0$.
A fundamental system of solutions to the differential equation
$(D_\mu + z^2 ) f = 0$ is explicitly given in terms of the 
modified Bessel functions $I_z, K_z$ by 
\begin{equation}
    x^{-1/2}\cdot K_z( \mu x), \quad x^{-1/2} \cdot I_z( \mu x).
\end{equation}    
Spectral problems and the analysis of the resolvent therefore
ultimately reduce to questions about the modified Bessel functions
and their asymptotic behavior. Besides the fairly standard
asymptotics for large arguments and fixed order resp. large order
and fixed arguments we will also need less standard \emph{uniform}
asymptotics for large order. The necessary facts about Bessel
functions are compiled in Section \ref{s.bessel}.

\subsubsection{The resolvent expansion, and meromorphic continuation of the
zeta-function}

The first step is to analyze the asymptotic expansion of the resolvent trace. 

\begin{theorem}\label{T.INTRO.Resolvent} Let $V\in X^\gamma L^1[a,\infty)$ as
in Theorem \ref{T.main} and suppose that a fixed (Dirichlet or generalized
Neumann) boundary condition for $H$ at $a$ is given. Then the resolvent
$(H(R_a) + z^2)\ii$ is trace class and there is an asymptotic expansion
\begin{equation}\label{eq.INTRO.resolvent.1}
  \Tr( H(R_a) + z^2)\ii = b_0 \cdot z\ii \cdot \log z + a_0\cdot z\ii 
  + a_1 \cdot z^{-2} + O(z^{-2-\delta}), \quad\text{ as } z\to\infty
\end{equation}
for some $\delta>0$. The constants $b_0, a_0, a_1$ do not depend on the
potential and not on $z$. Explicitly, $a_0 = \frac 12 \log \frac{2}{\mu a}$,
$b_0 = \frac 12$; for Dirichlet boundary conditions at $a$ we have 
$a_1 = \frac 14$ while for generalized Neumann conditions we have
$a_1 = -\frac 14$.

For the model operator $D_\mu$ ($V = 0$) there is a full asymptotic expansion
\begin{equation}\label{eq.INTRO.resolvent.model}
\Tr (D_\mu + z^2)^{-1} = \sum_{k=0}^{\infty} a_k(\mu) \cdot z^{-1-k}
+ \sum_{k=0}^{\infty} b_{2k}(\mu) \cdot z^{-1-2k}\cdot \log z,
\quad \text{ as } z\to\infty.
\end{equation} 
\end{theorem}
The full asymptotic expansion for the model operator is due to
the third named author \cite[Sec.~4]{Ver}, however without explicitly
specifying the first few coefficients.

The symbol $\regint$ will denote the Hadamard partie finie integral
which we will briefly review in Section \ref{ss.rli}.
The well-known formula, \cf \cite[(2.30)]{LesTol:DOD},
\begin{equation}\label{eq.INTRO.ZetaFunction}
     \zeta_H(s):=\sum_{\gl\in\spec H} \gl^{-s}= \frac{\sin \pi s}{\pi} 
     \cdot \regint_0^\infty x^{-s} \cdot \Tr(H+x)\ii dx
\end{equation}
relates the zeta-function of $H=H(R_a)$ to the resolvent trace and the
asymptotic expansion \Eqref{eq.INTRO.resolvent.1} implies that
$\zeta_H(s)$ has a meromorphic continuation to $\Re s > -\delta$
with $0$ being a regular point. Therefore one has
\begin{equation}\label{eq.INTRO.ZetaPrimeZero}
\log\detz H:=-\zeta_H'(0)=- \regint_0^\infty \Tr(H+z)\ii dz = -2
\regint_0^\infty z \cdot \Tr(H+z^2)\ii dz.
\end{equation}

A consequence of the expansion \Eqref{eq.INTRO.resolvent.1} is that as $z\to \infty$
\begin{equation}\label{eq.INTRO.LIMDet}
  \log\detz(H+z^2) = 2 \cdot b_0 \cdot z\cdot \log z
       + 2 \cdot (a_0 - b_0 ) \cdot z + 2\cdot a_1\cdot\log z + O(z^{-\delta}
       \log z).
\end{equation}
Note that there is no constant term, thus 
$\LIM\limits_{z\to\infty} \log\detz(H+z^2) = 0$
(Lemma \ref{L.gen.asymptotic-det}). $\LIM$ (regularized limit)
is a short hand for the constant term in the asymptotic expansion,
\cf Sec. \ref{ss.rli}.

\subsubsection{Weyl eigenvalue asymptotics } 
In this subsection we discuss the asymptotic behavior of the 
eigenvalue counting function $N(\lambda)$ for the cusp operator $H$.

Note that the presence of $z\ii\log z$ as the leading term in the resolvent trace 
asymptotics
in Theorem \ref{T.INTRO.Resolvent} above, distinguishes our case significantly from 
similar discussions of regular-singular Sturm--Liouville operators over a finite interval 
in \cite{Les:DRS} and \cite{LesVer}. There the singular potential in 
the Sturm--Liouville operator leads to logarithms in the resolvent trace 
asymptotics as well, however in contrast to our case, the logarithm does not 
appear in the leading term. 

The logarithmic leading term $z\ii\log z$ is obviously
a new phenomenon of our non-compact setting and has an important consequence 
for the Weyl asymptotics of the cusp operator $H(R_a)$. Indeed, by the resolvent 
trace expansion, one concludes that 
\begin{equation}
\zeta_H(s) - \frac{1}{\Gamma(s)} \left(\frac{-c_0}{\left(s-\frac{1}{2}\right)} + 
\frac{c_1}{\left(s-\frac{1}{2}\right)^2}\right) 
\end{equation}
is continuous for $\Re(s) \geq \frac{1}{2}$, where the constants $c_0$
and $c_1$ are determined explicitly by the coefficients in the resolvent 
trace asymptotics and in particular $c_1 = - \frac{b_0}{2 \Gamma\left(\frac{1}{2}\right)}$.
Now, by a Tauberian argument, \cf Shubin \cite[Problem 14.1 pp. 127]{Shu01} and 
Aramaki \cite{Ara83}, one concludes for the eigenvalue counting function 
\begin{equation}
N(\lambda) \sim \frac{\sqrt{\lambda} \log (\lambda)}{2\Gamma\left(1/ 2 \right)^2}, 
\quad \lambda \to \infty.
\end{equation}
Note that it is by no means straightforward to conclude a similar expansion 
for the eigenvalue counting function of the Laplace--Beltrami operator on cusps, 
which can be written as an infinite  direct sum of the cusp operators $H$. This is due to the 
non-uniform behaviour of the resolvent trace expansion in Theorem \ref{T.INTRO.Resolvent} as 
$\mu$ goes to infinity. A similar question has been studied in the joint work of the
second and third author \cite{LesVer2015}.

\subsubsection{Variation formula} The next step is to establish a variation
formula. Let us state it informally first: let $V_t$ be a (sufficiently nice)
one parameter family of potentials (satisfying the overall assumptions of
Theorem \ref{T.main} and depending differentiably on $t$) and denote by
$\phi_t,\psi_t$ be a normalized fundamental system of solutions to the
differential equation 
$H_t f =  D_\mu f + V_t f = 0$. Then
\begin{equation}\label{eq.INTRO.variation}
  \pl_t \log\detz H_t = \pl_t \log \bl p\cdot  W(\psi_t, \phi_t ) \br.
\end{equation}
This variation formula goes back to Levit and Smilansky
\cite{LevSmi1977} for the situation of Theorem \ref{T.BFK}.
For $V_t = t^2$ being the resolvent parameter of the model
operator it is due to the third named author \cite[Sec.~5]{Ver}. 
In Section \ref{s.variation-model} we will present 
an expanded version which includes some important details.
For the general case we investigate the dependence of the asymptotic behavior of a fundamental
system of solutions at infinity on the parameter $t$ and we prove a B\^ocher Theorem for $H+\nu^2$ in
Section \ref{Bocher}. Finally, we analyze the asymptotic expansion of the resolvent trace of the perturbed 
operator and prove the Theorem \ref{T.INTRO.Resolvent} in Section \ref{s.regdet_perturbed} and 
prove the final result in Section \ref{s.regdet}.
 
\mpar{TODO: write more what needs to be done here,
e.g. Bocher Thm and limit of $p W(\psi,\dot\phi)$ ...}

We apply \Eqref{eq.INTRO.variation} to $V_z = V +z^2$
and obtain in view of \Eqref{eq.INTRO.LIMDet}
\begin{equation}
  \detz (H  +z_0^2)= p\cdot W(\psi_{z_0},\phi_{z_0}) \cdot
     \exp\Bl -\LIM_{z\to\infty} \log\bl p\cdot W(\psi_z, \phi_z) \br\Br.
\end{equation}     
For general one dimensional elliptic differential operators on
a finite interval this formula was established in \cite[Thm.~3.3]{LesTol:DOD}.

It remains to compute the constant $\LIM\limits_{z\to\infty} \log \bl p\cdot W(\psi_z,
\phi_z)\br$.  The variation formula \Eqref{eq.INTRO.variation} shows that this
constant is independent of the potential $V$.  Therefore, it suffices to
compute it for the model operator.  In \cite{BFK:DEB} and \cite{Les:DRS} this
is done by proving another formula for the variation of generalized Neumann
conditions and then finally by computing explicit examples. Namely, on a
finite interval the $\zeta$--determinant for $-\frac{d^2}{dx^2}$ (with
Dirichlet or Neumann boundary conditions) can explicitly be expressed in terms
of the Riemann $\zeta$--function for which the derivative at $0$ is known
(Lerch's formula).  The case of a regular singular operator on a finite
interval can also be reduced to this case; alternatively one can take
advantage of the fact that the spectrum of the Jacobi differential operator is
explicitly known \cite{Les:DRS}.

For our model operator $D_\mu$ we need to employ a different strategy as we do
not know the spectrum of any self-adjoint
realization of \Eqref{eq.INTRO.cusp.op} for any parameter value $\mu$. However, for each 
boundary condition
the normalized fundamental system $\phi_z,\psi_z$ can explicitly be expressed
in terms of the modified Bessel functions. Consequently, the asymptotic behavior
of $\log \bl p\cdot W(\psi_z,\phi_z) \br$ can be studied with the help of the
known asymptotics of the modified Bessel functions.

\subsection*{Acknowledgements} The first author was partially supported by
grant FAPESP 2015/01923-0 and is grateful to Bonn University for hospitality.
The third author thanks Werner M\"uller for valuable discussions on the
geometry of manifolds with cusps and is grateful to Bonn and M\"unster
University for hospitality.  All authors gratefully acknowledge the support of
the Hausdorff Center for Mathematics.

\section{Generalities: Regularized limits, $\zeta$--determinants, \\ and
Wronskians of Sturm--Liouville operators}
\label{s.generalities}

For the convenience of the reader and to fix some notation we collect
here some general facts on regularized limits, zeta-determinants and
Wronskians of Sturm--Liouville operators, \cf also \cite{Les:DRS},
\cite{LesTol:DOD}, \cite[Sec.~1]{LesVer}, \cite[Sec.~1]{LesVer2015}
and the references therein.

\newcommand{\ddx}{\frac{d}{dx}}
\subsection{Regularized limits and integrals} \label{ss.rli}
Let $f:(0,\infty)\to \C$ be a function with a (partial) asymptotic expansion 
\begin{equation}\label{eq.gen.reg-limit}
f(x) \sim \sum_{j=1}^{N-1} \sum_{k=0}^{M_j} a_{jk}x^{\ga_j} \log^k(x) +
             \sum_{k=0}^{M_0} a_{0k} \log^k(x) +  f_N(x), \ \quad \ x\ge x_0>0,
\end{equation}
where $\ga_j\in\C$ are ordered with decreasing real part
and the remainder $f_N(x)=o(1)$ (Landau notation) as $x\to \infty$. 
Then we define its \emph{regularized limit} as $x\to \infty$ by
\begin{equation}
\LIM_{x\to \infty}f(x) :=a_{00}.
\end{equation}
If $f$ has an expansion of the form \Eqref{eq.gen.reg-limit} as $x\to 0$ then
the regularized limit as $x\to 0$ is defined accordingly. 

If $f$ is locally integrable and the remainder $f_N\in L^1[1,\infty)$ even 
integrable, the integral $\int_1^R f(x)dx$ also admits an asymptotic 
expansion of the form \Eqref{eq.gen.reg-limit} and one defines 
the \emph{regularized integral}  as 
\begin{equation}\label{EqRegInt1}
\regint_1^\infty f(x)dx := \LIM_{R\to \infty}\int_1^R f(x) dx.
\end{equation}
Similarly, 
$\regint_0^1 f(x)dx := \LIM\limits_{\eps\to 0}\int_\eps^1 f(x) dx$,
if this regularized limit exists. 

$\regint$ is a linear functional extending the ordinary integral.
However, it has some pathologies. \textit{E.~g}., the formula for changing variables 
$x\mapsto \gl\cdot x$ in the integral has correction terms,
\cite[Lemma 1.1]{LesVer2015}. Relevant for us will be the behavior
under translations. Namely, assuming that $f$ is locally integrable
with remainder $f_N\in L^1[1,\infty)$, consider for $x>0$
\begin{equation}
  \begin{split}
    \regint_0^\infty f(x+t) dt & = \LIM_{R\to\infty} \int_0^R f(x+t) dt\\
       & = \LIM_{R\to\infty}\Bl \int_x^R f(t) dt + \int_R^{R+x} f(t) dt\Br\\
       & = \regint_x^\infty f(t) dt + \LIM_{R\to\infty} \int_R^{R+x} f(t) dt.
  \end{split}
\end{equation}
In general $\LIM_{R\to\infty} \int_R^{R+x} f(t) dt \not=0$, \cf
the discussion after Lemma 2.2 in \cite{LesTol:DOD}.
However it vanishes whenever there are no terms of
the form $x^\ga \log^k x$ with $\ga\in \Z_+\subset \C \backslash \{0\}$ in the expansion
\Eqref{eq.gen.reg-limit}. For later reference we record

\begin{lemma} \label{L.gen.1}
Let $f:(0,\infty)\to\C$ be locally integrable with an asymptotic
expansion as in \Eqref{eq.gen.reg-limit} where $\ga_j\not\in\Z_+$
and $f_N\in L^1[1,\infty)$. Then for all $x>0$
\begin{equation} \label{eq.gen.reg-limit1}
  \regint_0^\infty f(x+t) dt = \regint_x^\infty f(t) dt,
\end{equation}
in particular $\LIM\limits_{x\to\infty} \regint_0^\infty f(x+t) dt = 0$.
\end{lemma}
\begin{proof} The last claim is a consequence of
the identity \Eqref{eq.gen.reg-limit1}, as 
\begin{equation}\label{eq.gen.reg-limit2}
  \LIM_{x\to\infty} \regint_x^\infty f(t) dt =
  \regint_1^\infty f(t)dt - \LIM_{x\to\infty} \int_1^x f(t) dt = 0.
\end{equation}
The identity \Eqref{eq.gen.reg-limit2} follows from the very definition of the 
regularized integral,
regardless of the values of the exponents $\ga_j$.
\end{proof}

\subsection{Zeta-regularized determinants}
\label{ss.gen.ZetaDeterminant}
Let $H>0$ be a self-adjoint positive operator acting on some Hilbert space. We
assume that the resolvent of $H$ is trace class,  and that for $z\ge 0$ we
have
\begin{equation}\label{eq.gen.DetExp}
  \Tr( H+z)\ii= \sum_{-1-\delta < \Re\ga <0} z^\ga\cdot P_\ga(\log
  z)+O(z^{-1-\delta}),\quad\text{ as } z\to\infty,
\end{equation} 
with polynomials $P_\ga(t)\in\C[t]$, $P_\ga=0$ for all but finitely
many $\ga$. Moreover, we assume that $P_{-1}$ is of degree $0$, that is
there are no terms of the form $z^{-1}\cdot \log^k z$ with $k\ge 1$.
For $1<\Re s<2$ the \emph{zeta-function ($\zeta$--function)} of $H$ is given by, 
\cf \cite[(2.30)]{LesTol:DOD},  
\begin{equation}\label{eq.gen.ZetaFunction}
     \zeta_H(s):=\sum_{\gl\in\spec H} \gl^{-s}= \frac{\sin \pi s}{\pi} 
     \cdot \regint_0^\infty x^{-s} \cdot \Tr(H+x)\ii dx.
\end{equation}
From the asymptotic expansion \Eqref{eq.gen.DetExp} one deduces that
$\zeta_H(s)$ extends meromorphically to the half plane $\Re s > -\delta$,
\cite[Lemma 2.1]{LesTol:DOD}.  The identity \Eqref{eq.gen.ZetaFunction}
persists except for the poles of the function $s\mapsto \frac{\pi}{\sin \pi s}
\zeta_H(s)$. From the assumption that $\deg P_{-1} =0$ in
\Eqref{eq.gen.DetExp} it follows that $\zeta_H$ is regular at $s=0$ and one
puts
\begin{equation}\label{eq.gen.ZetaPrimeZero}
\log\detz H:=-\zeta_H'(0)=- \regint_0^\infty \Tr(H+z)\ii dz = -2
\regint_0^\infty z \cdot \Tr(H+z^2)\ii dz.
\end{equation}
$\detz H$ is called the \emph{zeta-determinant ($\zeta$-determinant)} or 
\emph{zeta-regularized 
determinant} of $H$. For
non-invertible $H$ one puts $\detz H=0$. With this setting the function
$z\mapsto \detz (H+z)$ is an entire holomorphic function with zeros exactly
at the eigenvalues of $-H$. The multiplicity of a zero $z$ equals the
algebraic multiplicity of the eigenvalue $z$.

If $P_{-1}$ is a higher order polynomial then $\zeta_H$ has poles
at $0$. One still could define $-\log\detz H$ to be the coefficient
of $s$ in the Laurent expansion about $0$ of $\zeta_H(s)$. However,
in this case the relation $\log\detz H = - \regint_0^\infty \Tr(H+z)\ii dz$
would not hold any more. 

\begin{lemma}\label{L.gen.asymptotic-det}
 Let $H$ be a bounded below self-adjoint operator in some Hilbert space.
Assume that the resolvent is trace class and that, as $z\to\infty$,
the expansion \Eqref{eq.gen.DetExp} holds.
Then, as $z\to\infty$ we have an asymptotic expansion
\begin{equation}
  \log\detz (H+z) = \sum_{-1-\delta < \Re \ga < 0} z^{\ga+1}\cdot
      Q_\ga(\log z) + O( z^{-\delta} )
\end{equation}
with polynomials $Q_\ga$ satisfying $Q_\ga' = -(\ga+1) Q_\ga + P_\ga$.
Moreover, $Q_{-1}(\log x) = P_{-1}(0) \cdot\log x$, in particular
$\LIM\limits_{z\to\infty} \log\detz(H+z) = 0$.
\end{lemma}
\begin{proof} This follows in a straightforward fashion
  from the definition of the regularized integral, the relation
  \Eqref{eq.gen.ZetaPrimeZero} and Lemma \ref{L.gen.1}.
  For details, \cf \cite[Lemma 2.2]{LesTol:DOD}.
\end{proof}

With regard to Lemma \ref{L.gen.1} we emphasize that under the assumptions
of the previous Lemma we have for the zeta-determinant of $H+\nu^2$
\begin{equation}
  \log\detz(H+\nu^2) = - \regint_{\nu^2}^\infty \Tr(H+z)\ii dz 
				 = -2 \regint_{\nu}^\infty z\cdot \Tr(H+z^2)\ii dz.
\end{equation}

The result
\begin{equation}
\LIM_{z\to\infty} \log \detz (H+z) = 0
\end{equation}
contains the main result of \cite{Fri89} as a special case. Namely, one has for
$z\geq 0$ and invertible $H$
\begin{equation}
\detf(I + z H^{-1}) = \frac{\detz(H+z)}{\detz H},
\end{equation}
where $\detf$ denotes the Fredholm determinant.
This follows immediately from the fact that the left hand side and the right hand side 
have the same $z$-derivatives and that they coincide at $z=0$. Therefore, we have as 
$z\to\infty$,
\begin{equation}
\begin{split}
\log \detf (I+z H^{-1}) &= \log \detz(H+z) - \log \detz H\\
&=\sum_{-1-\delta < \Re \ga < 0} z^{\ga+1}\cdot
Q_\ga(\log z)  - \log \detz H + O( z^{-\delta} ).
\end{split}
\end{equation}
In particular, 
\begin{equation}
\LIM_{z\to\infty} \log\detf (I+z H^{-1}) = -\log \detz H,
\end{equation}
which is Friedlander's \cite{Fri89} formula.

\subsection{Wronskians and their variation}
\label{ss.wronskians}

Let 
\begin{equation} \label{eq.gen.diffop}
  H_0 = -\frac{d}{dx}\Bigl( p(x) \frac{d}{dx}\cdot \Bigr) + V_0(x),
\end{equation}
be a differential operator on the interval 
$(a,\infty)$, $a>0 $, with a positive continuous 
function $p\in C[a,\infty)$, $p(x)>0$, and locally integrable
potential $V_0\in L^1_{\loc}[a,\infty)$. Furthermore, we assume that
$a$ is a regular point and that $\infty$ is in the limit point case for
$H_0$. We fix a self-adjoint boundary condition $R_a f = 0$ at $a$
and assume that $H_0$ with this boundary condition is invertible.

Let $\psi$, $\phi$ be a fundamental system of solutions to the
differential equation $H_0f = 0$ with $R_a \phi=0$ and $\psi\in L^2[a,\infty)$.
Then the \emph{Wronskian}
\begin{equation} \label{eq.gen.wronskian-1}
    p \cdot W(\psi,\phi) = p \cdot \Bigl( \psi \cdot \phi' - \psi' \cdot\phi\Br
\end{equation}  is constant. The Schwartz kernel (Green function) of $H_0^{-1}$ is given 
by
\begin{equation}\label{eq.gen.green-kernel}
  G(x,y) = \frac{1}{p \cdot W(\psi,\phi)}\cdot 
       \begin{cases}
            \phi(x) \cdot \psi(y),  &  x\leq y,  \\
            \psi(x) \cdot \phi(y),  &  y\leq x.
        \end{cases}
\end{equation}
\details{ 
In fact, given $g\in L^2[a,\infty)$ put
\begin{equation*}
     f(x) = \phi(x) \cdot \int_{x}^{\infty} \psi(y)g(y) dy 
           + \psi(x)\cdot \int_{a}^{x}\phi(y)g(y) dy,
\end{equation*}
a direct calculations shows that
\begin{equation*}
    H_0f = -p\cdot W(\phi,\psi) \cdot g = p\cdot  W(\psi,\phi)\cdot g.
\end{equation*}
} 

Suppose now that $V_0$ depends differentiably on a parameter $t$. 
Assume that $\psi_t$ and $\phi_t$ are solutions as above depending 
differentiably on $t$. Denote the differentiation by $t$ by a dot 
decorator, \textit{e.g.} $\pl_t \phi =: \dot\phi$ and differentiation by
$x$ by a ' decorator, \textit{e.g.} $\pl_x \phi =: \phi'$.
Differentiate the differential equation 
$-\bl p \cdot  \psi '\br' + V_0\cdot \psi = 0$
by $t$ to obtain
\begin{equation}
  -\bl p \cdot  \dot\psi' \br' + V_0 \cdot \dot\psi(x) = -\dot V_0 \cdot \psi,
\end{equation} 
and similarly for $\phi$. Hence 
\begin{equation}
  p\cdot W(\psi,\phi) \cdot \dot V_0(x) \cdot  G(x,x) = \dot V_0 (x) \cdot  \phi(x)\cdot \psi(x),
\end{equation}
thus 
\begin{equation} \label{eq.d.pert.W}
\begin{split}
\dot V_0 \cdot \phi \cdot\psi
    & = \bl \dot V_0 \cdot \phi\br \cdot \psi
        = \Bl \pl_x \bl p \cdot \pl_x \dot \phi \br - V_0 \dot\phi \Br \cdot\psi\\
    & = \psi \cdot \pl_x \bl p \cdot \pl_x \dot\phi \br 
         - \dot\phi \cdot \pl_x \bl p\cdot \pl_x \psi \br \\
    & = \frac{d}{dx}\Bl p \cdot \bl \pl_x\dot{\phi} \br \cdot \psi 
           - p \cdot \dot\phi \cdot \pl_x \psi \Br 
	    = \frac{d}{dx}\Bl p\cdot W(\psi,\dot\phi)\Br.
\end{split}
\end{equation} 
Thus if the operator $\dot{V}_0 H_0\ii$ is trace class, then
    \begin{multline}
   p\cdot W(\psi,\phi) \cdot \Tr\bl \dot V_0 H_0^{-1}\br 
    = \int_a^\infty  \dot V_0 \cdot \phi \cdot\psi \\
    = \left.p(x)\cdot W(\psi,\dot{\phi})(x)\right|_{x=a}^{x=\infty} = 
       \left.p(x)\cdot W(\phi,\dot{\psi})(x)\right|_{x=a}^{x=\infty},
    \end{multline}   
where the second equation follows by exchanging $\phi$ and $\psi$ in
the calculation. Note that the trace class property plus the
regularity at $a$ imply the existence of the limit
$\lim\limits_{x\to\infty} p(x)W(\psi,\dot{\phi})(x)$.
Furthermore, 
\begin{equation}
\pl_t \Bl p \cdot W (\psi,\phi)\Br
  = p\cdot \Bl W(\dot\psi,\phi)+W(\psi,\dot\phi)\Br.
\end{equation}
By \Eqref{eq.gen.wronskian-1} the Wronskian $p\cdot W(\psi,\phi)$ is
a constant function in $x$. Moreover, if at the regular end $\phi$ is normalized, 
then $\dot{\phi}(a) = \frac{d}{dx}\dot{\phi}(a) = 0$,
hence $W(\dot{\phi},\psi)(a) = 0$.
Thus altogether we have proved
\begin{prop} \label{GEN-Trace-Formula} Let $H_0$ be the
differential operator \Eqref{eq.gen.diffop} and assume that
$(V_{0,t})_t$ depends differentiably on a parameter $t$.
Furthermore, let $\phi_t,\psi_t$ be a fundamental system
of solutions such that $\phi_t$ is normalized at $a$ and
$\psi_t\in L^2[a,\infty)$; assume that $\phi_t, \psi_t$
depend differentiably on $t$. Then 
\begin{equation} \label{GEN-Wronskian-2}
p(a)\cdot W(\phi,\psi)(a) = \partial_t\bl p \cdot W(\phi,\psi)\br
  = - \pl_t\bl p\cdot W(\psi,\phi)\br.
\end{equation}
Furthermore, if $\dot V_0 H_0\ii$ is trace class and if
\[
\lim_{x\to\infty} p(x)W(\phi,\dot{\psi})(x) = 0
\] 
then
\begin{equation}
\begin{split}
  \Tr\bl \dot{V}_0 H_0\ii \br
    & = \frac{1}{p \cdot W(\psi,\phi)} \left. 
                      p \cdot W(\phi,\dot{\psi})\right|_a^\infty\\
    & = \frac{1}{p \cdot W(\psi,\phi)} \pl_t\Bl p\cdot W(\psi,\phi)\Br\\
    & = \pl_t \log \Bl p\cdot  W(\psi,\phi)\Br.
\end{split}
\end{equation}
\end{prop}

\subsection{Perturbative solutions, B\^ocher's Theorem}
\label{ss.Pert-Boch}
Let $H_0$ be as in \Eqref{eq.gen.diffop}. We do not impose any boundary
condition in this subsection. Let $\phi,\psi$ be any fundamental system of
solutions to the differential equation $H_0 f = 0$. Without loss of
generality, we may assume their Wronskian equals $1$, \ie $p\cdot
W(\psi,\phi) = 1$. The solution formula for the inhomogeneous equation 
$H_0 u = v$ then reads
\begin{equation}\label{GEN-VarConst}
  \begin{split}
     u(x)  & = c_1\cdot \psi(x)+c_2\cdot \phi(x) - \psi(x)\cdot
    \int_{x}^{\infty} \phi(y)\cdot v(y) dy \\
    &\qquad +\phi(x)\cdot \int_{x}^{\infty}  \psi(y)\cdot v(y) dx,
\end{split} 
\end{equation}
if, for all $x\in(a,\infty)$,
\begin{equation}
  \int_{x}^{\infty}|\psi(y) v(y)|dy <\infty,\quad \text{ and } \quad
  \int_{x}^{\infty}|\phi(y) v(y)| dx <\infty.
\end{equation}  

This formula may be used to find a fundamental system of solutions with prescribed
asymptotics for perturbations $H = H_0 + V$ of $H_0$. Here we just present the
general pattern. We will apply this to our concrete model operator in
section \ref{Bocher} below. For a solution of $Hf = 0$ we make the
perturbative Ansatz $h_1(x) = \psi(x)(1+f_1(x))$. This
leads to the non-homogeneous equation
\begin{equation}
    H_0\bl \psi \cdot f_1 \br = - V \cdot \psi \cdot \bl 1 + f_1\br.
\end{equation}
We denote by $L$ the integral operator with kernel
\begin{equation}  \label{kernelL}
    L(x,y) = p(y)\cdot \psi^2(y) \cdot\Bl \frac{\phi(y)}{\psi(y)} -
    \frac{\phi(x)}{\psi(x)}\Br.
\end{equation} 
Then writing $V =: p\cdot W$, \Eqref{GEN-VarConst} leads to the ansatz
\begin{equation}\label{GEN-f-1}
f_1(x) = \int_{x}^{\infty}L(x,y)\cdot W(y)\cdot (1+f_1(y)) dy.
\end{equation} 
Note that, since
\begin{equation}\label{q.psi.phi}
\left(\frac{\phi}{\psi}\right)' = \frac{1}{p\cdot\psi^2},
\end{equation} 
we find
\begin{equation}\label{Deriv.L}
(Lf)'(x) = -\frac{1}{p(x)\cdot \psi^2(x)} \int_{x}^{\infty} p(y)\cdot
\psi^2(y)\cdot f(y)dy.
\end{equation}
Now consider the following assumptions:
\begin{align}
\sup_{a\leq x \leq y \leq \infty} |L(x,y)| &< \infty, \label{Assump1}\\
\lim_{x\to \infty} \psi(x) &= 0, \label{Assump2}\\
\frac{\psi(x)}{\psi'(x)} &= O(1),\;\text{as}\; x\to\infty,\label{Assump3}\\
\sup_{a\leq x \leq y \leq \infty} 
\frac{p(y)\cdot \psi(y)}{p(x)\cdot \psi(x)}&<\infty,\label{Assump4}\\
\int_{1}^{x}\frac{\phi(y)}{\psi(y)}dy &= O\left(\frac{\phi(x)}{\psi(x)}\right),\;\text{as}\; 
x\to\infty,\;{\rm and} \lim_{x\to\infty}\frac{\phi(x)}{\psi(x)} = \infty\label{Assump5}.
\end{align}

Denote by $C_b^k[a,\infty)$ the Banach space of $k$-times
  continuously differentiable functions with bounded derivatives up to
  order $k$. I.~e.
\begin{equation}
    \| f \|_{C_b^k} := \sum_{j = 0}^k \sup_{a\le x <\infty} | f^{(k)}(x) |.
\end{equation}
Furthermore, the space $X^{-\gamma}C_b^k[a,\infty)$
is a Banach space with norm $\|f\|_{C_b^k,\gamma} := \| X^\gamma
f\|_{C_b^k}$. \mpar{this definition breaks if $a\le 0$} We write $C_b[a,\infty) := 
C^0_b[a,\infty)$ for the Banach space of bounded continuous functions. We also write 
$C^k_\bullet [a,\infty)\subset C^k_b [a,\infty)$ for the subspace of bounded $k$ times 
continuously differentiable functions which converge to zero at infinity along with their 
derivatives up to order $k$.

\begin{lemma}\label{L.op}
	Let $W$ be a function in $L^1(a,\infty)$. Then for each $\gamma \geq 0$ the 
	Volterra 
	operator $LW$ maps $X^{-\gamma} C_b[a,\infty)$ continuously into 
	$X^{\gamma}C^1_{\bullet}[a,\infty)$. Moreover, as an operator in 
	$X^{-\gamma}C_b[a,\infty)$ it has spectral radius zero. Finally, the map 
	\begin{equation}
	L^1[a,\infty) \ni W \mapsto LW \in \mathcal{L}(X^{-\gamma}C_{b}[a,\infty))
	\end{equation}
	is continuous from $L^1[a,\infty)$ into the bounded linear operators 
	on $X^{-\gamma}C_{b}[a,\infty)$, where for $f\in 
	X^{-\gamma}C_b[a,\infty)$
	\begin{equation}
	(L W f) (x) = \int_{x}^{\infty} L(x,y) W(y) f(y) dy.
	\end{equation}
\end{lemma}

\begin{proof}
	Clearly by \Eqref{Assump1},
	\begin{equation}
	\begin{split}
	|(LWf)(x)| &\leq \int_{x}^{\infty} |W(y)|\cdot |f(y)| dy\\
	&\leq x^{-\gamma} \cdot \|f\|_{\gamma} \cdot \int_{x}^{\infty} |W(y)| dy
	\end{split}
	\end{equation} and inductively
	\begin{equation}
	\bigl|((LW)^nf)(x)\bigr|\leq x^{-\gamma} \cdot \frac{\|f\|_\gamma}{n!}\cdot 
	\left(\int_x^\infty 
	|W(y)|dy\right)^n.
	\end{equation} This proves that $LW$ maps $X^{-\gamma}C_b[a,\infty)$ continuously 
	into $C_{\bullet,\gamma}[a,\infty)$ and that, as an operator in $X^\gamma 
	C_b[a,\infty)$ it	has spectral radius zero. It also implies the last sentence of 
	the Lemma.
	
	Furthermore, from \Eqref{Deriv.L} we infer
	\begin{equation}
	\begin{split}
	|(LWf)'(x)| &= \left|\frac{1}{p(x)\cdot \psi^2(x)} \int_{x}^{\infty} 
	p(y)\cdot \psi^2(y)\cdot W(y)\cdot f(y)dy\right|\\
	&\leq C_1 x^{-\gamma} \cdot \|f\|_{\gamma} \cdot \int_{x}^{\infty} |W(y)|dy\\
	&\leq C_2 x^{-\gamma} \|f\|_{\gamma},
	\end{split}
	\end{equation} by \Eqref{Assump4}.
\end{proof}

\begin{theorem}
Suppose that there exists $x_0\in[a,\infty)$ such that $\psi(x)\not = 0$ for
$x\geq x_0$.  Under the assumptions \Eqref{Assump1}, \eqref{Assump2},
\eqref{Assump3}, \eqref{Assump4} and \eqref{Assump5} the perturbed operator 
\begin{equation}
	H = H_0 + V = H_0 + p\cdot W
\end{equation} 
has a fundamental system of solutions $h_1$, $h_2$ of the form
\begin{equation}
	h_1(x) = \psi(x) \cdot g_1(x), \qquad h_2(x) = \phi(x) \cdot g_2(x),
\end{equation} 
with $g_j\in C_b[a,\infty)$, $\lim\limits_{x\to\infty} g_j(x) = 1$, $j=1,2$. Furthermore,
\begin{equation}
	h'_1(x) = \psi'(x) \cdot \tilde g_1(x), \qquad h'_2(x) = \phi'(x) \cdot \tilde g_2(x),
\end{equation} 
with $\tilde g_j \in C_b[a,\infty)$, $\lim\limits_{x\to\infty} \tilde g_j(x) = 1$, 
$j=1,2$ and $p(x)W(h_1,h_2)(x) = 1$.
\end{theorem}

\begin{proof} We denote the
	constant function equal to one by a bold number one, ${\bf 1}(x)  = 1$ for all $x$.
	By equation \Eqref{GEN-f-1} and the properties of Lemma \ref{L.op} we conclude that
	\begin{equation}\label{f1.LW}
	f_1(x) = (I-LW)^{-1}(LW{\bf 1})(x)
	\end{equation} is in $C^1_\bullet[a,\infty)$, hence
	\begin{equation}
	h_1(x) = \psi(x)\cdot(1+f_1(x))
	\end{equation}
	has the claimed properties. Note that
	\begin{equation}
	h_1'(x) = \psi'(x)\cdot \bl 1+f_1(x)+\frac{\psi(x)}{\psi'(x)}\cdot f'_1(x)\br,
	\end{equation} 
	and by assumption \Eqref{Assump3}, we conclude that $\tilde g_1 = 
	1+f_1+\frac{\psi}{\psi'} \cdot f'_1\in 
	C_b[a,\infty)$ and
	\begin{equation}
	\lim_{x\to\infty} \tilde g_1(x) = 1.
	\end{equation}
	
	For the second solution one finds
	\begin{equation}
	h_2(x) = c(x)\cdot h_1(x),
	\end{equation}
	with
	\begin{equation}
	c(x) = \int_{x_0}^{x} p(y)^{-1} \cdot h_1(y)^{-2} dy.
	\end{equation}
	From \Eqref{q.psi.phi} we infer by integrating by parts
	\begin{equation}
	\begin{split}
	c(x) &= \int_{x_0}^{x} \frac{1}{p(y) \cdot \psi(y)^2}(1+f_1(y))^{-2} dy\\
		 &=\frac{\phi(y)}{\psi(y)} (1+f_1(y))^{-2}\Bigr|_{x_0}^x + 2 \int_{x_0}^{x} 
			\frac{\phi(y)}{\psi(y)}\cdot\frac{g_1'(y)}{g_1(y)^3} dy.
	\end{split}
	\end{equation}
	Thus,
	\begin{equation}
	\begin{split}
	g_2(x) =& c(x)\frac{h_1(x)}{\phi(x)}\\
		   =& \frac{1}{g_1(x)} - \frac{\psi(x)}{\phi(x)} \cdot
		   \frac{\psi(x_0)}{\phi(x_0)}\cdot\frac{g_1(x)}{g_1(x_0)^{2}}
		   + 2g_1(x)\cdot \frac{\psi(x)}{\phi(x)}\int_{x_0}^{x}\frac{\phi(y)}{\psi(y)} 
		   \cdot
		   \frac{g_1'(y)}{g_1(y)^3} dy.
	\end{split}
	\end{equation} By assumption \Eqref{Assump5},
	\begin{equation}
	\Bigr|\frac{\psi(x)}{\phi(x)}\int_{x_0}^{x}\frac{\phi(y)}{\psi(y)} dy\Bigl| \leq C,
	\end{equation} and since $\lim\limits_{x\to\infty }g'_1(x) = 0$ by Lemma \ref{L.op},
	we obtain $\lim\limits_{x\to\infty }g_2(x) = 1$.
	
	Furthermore, direct computation shows
	\begin{equation}
	p(x)\cdot W(h_1,h_2)(x) = 1.\qedhere
	\end{equation}
\end{proof}

\subsection{Comparison of zeta-determinants for Dirichlet and generalized Neumann
boundary conditions}
\label{ss.comparison}

In this subsection we will show that under very mild assumptions it is
possible to compute the relative $\zeta$--determinant of the operator $H_0$ in
\Eqref{eq.gen.diffop} with respect to two different boundary conditions at the
left endpoint. Under the general assumptions of Sec. \ref{ss.wronskians} we
consider the Dirichlet boundary condition at $a$, $f(a) = 0$, and a generalized Neumann 
boundary condition at $a$, $R_af = f'(a) + \alpha \cdot f(a)$. Furthermore, let
$\phi=\phi_z,\psi=\phi_z$ be a fundamental system of solutions to the equation
$(H+z^2)u = 0$ such that $\phi(a) = 0, \phi'(a) = p(a)^{-\frac{3}{4}}$ (that
is, $\phi$ is normalized at $a$) and such that $\psi\in L^2[a,\infty)$.  We
suppress the $z$--dependence from the notation.  Furthermore, we consider
$z$ such that $H+z^2$ with both Dirichlet and the generalized Neumann boundary condition 
are invertible. This is certainly the case if $z$ is large enough.
Consequently, $\psi(a)\not=0\not=R_a\psi$. Denote by $H_D$ the self-adjoint extension of 
$H$ with Dirichlet boundary condition and by $H_\ga$ the self-adjoint extension of $H$ 
with generalized boundary condition $R_a$.

For the normalized solution $\phi_{\alpha} = \phi_{\alpha,z}$
satisfying the generalized Neumann condition we make the Ansatz
\begin{equation}\label{eq.DirNeu1}
  \begin{split}
    \phi_\alpha    & = \lambda_\alpha \cdot \phi + \mu_\alpha \cdot \psi\\
    \phi_\alpha(a) & = p(a)^{-\frac 14}, \quad \phi'_\alpha(a)+\alpha\cdot \phi_\alpha(a)=0
  \end{split}  
\end{equation}
and find
\begin{equation}\label{eq.DirNeu2}
	\begin{split}
	\mu_{\alpha}     & = \psi(a)^{-1}\cdot p(a)^{-\frac 14} \\
	\lambda_{\alpha} & = -\mu_{\alpha} \cdot \frac{R_a\psi}{\phi'(a)} 
				= -\sqrt{p(a)}\cdot \Bl \alpha + 
				\frac{\psi'(a)}{\psi(a)}\Br.
	\end{split}
	\end{equation}
Note that $\mu_\ga$ is independent of $\ga$ while $\pl_\ga \gl_\ga = -\sqrt{p(a)}$.

\details{
\begin{proof}
	\begin{equation}
	p(a)^{-\frac{1}{4}} = \psi_{\alpha}(a) = \lambda_\alpha \cdot \phi(a)+\mu_\alpha 
	\cdot \psi(a) 
	\Rightarrow
	\mu_\alpha = p(a)^{-\frac{1}{4}} \cdot \psi(a)^{-1},
	\end{equation}
	and
	\begin{equation}
	\begin{split}
	\lambda_\alpha \cdot \phi'(a) + \mu_\alpha \cdot \psi'(a) &= \phi'_\alpha(a) \\
			&= -\alpha \cdot \phi_\alpha(a)\\
			&= -\alpha \cdot p(a)^{-\frac{1}{4}}\\
			&= -\alpha \cdot \mu_\alpha \cdot \psi(a) \Rightarrow\\
	\lambda_\alpha \phi'(a) &= -\mu_\alpha \bl \alpha \cdot \psi(a) + \psi'(a) \br 
	= -\mu_\alpha \cdot R\psi.
	\end{split}
	\end{equation} Thus
	\begin{equation}
	\begin{split}
	\lambda_\alpha &= - \mu_\alpha \cdot\frac{R_a\psi}{\phi'(a)}\\
	&= - \frac{1}{p(a)^{\frac{1}{4}} \cdot \psi (a)} \cdot 
							\frac{\psi'(a)+\alpha\cdot\psi (a)}{p(a)^{-\frac{3}{4}}}\\
	&= -\sqrt{p(a)}\cdot \Bl \alpha + \frac{\psi'(a)}{\psi(a)}\Br. 
	\end{split}
	\end{equation}
\end{proof}
} 
Furthermore, we have for the Wronskians
\begin{align}
  p\cdot W(\psi,\phi ) & = p(a) \cdot \psi(a)\cdot\phi'(a) = \mu_\ga\ii, 
  \label{eq.Wronsk.sec2.5.1}\\
  p\cdot W(\psi_\ga,\phi_\ga) &= \gl_\ga \cdot p\cdot W(\psi,\phi) = 
  \frac{\gl_\ga}{\mu_\ga}.\label{eq.Wronski.sec2.5.2}
\end{align}

\begin{lemma} For $z\ge 0$ such that $\psi(a)=\psi_z(a)\not = 0$ we have 
  \begin{equation}
    \int_a^\infty \psi_z(y)^2 dy = \frac{1}{2z\cdot \mu_{\ga,z}^2} \frac{d}{dz}
    \gl_{\ga,z}.
  \end{equation}
\end{lemma}

\begin{proof} We use the formula for the Green function
\Eqref{eq.gen.green-kernel} of the operator
  $H_\ga+z^2$ and obtain for $z\ge 0$ such that $-z^2 $ is in the resolvent set
  (and still suppressing the index $z$ where appropriate)
\begin{equation}
\begin{split}
\int_{a}^{\infty} \psi_z(y)^2 dy 
&= \bl p\cdot W(\psi,\phi_{\alpha}) \br^2 
		\int_{a}^{\infty} \Bl \frac{\psi_z(y)}{p\cdot W(\psi,\phi_\alpha)} \Br^2 dy\\
	&=\bl p\cdot W(\psi,\phi_\alpha) \br^2 \cdot \phi_\alpha(a)^{-2}
		\int_{a}^{\infty} \Bl \frac{\psi_z(y) \phi_\alpha(a)}{p\cdot W(\psi,\phi_\alpha)} 
		\Br^2 dy\\
	&=\bl p\cdot W(\psi,\phi_\alpha) \br^2 \cdot \phi_\alpha(a)^{-2}
		\int_a^{\infty} \Bl (H_\alpha + z^2)^{-1}(a,y) \Br^2 dy\\
	&=\bl p\cdot W(\psi,\phi_\alpha) \br^2 \cdot \phi_\alpha(a)^{-2}
		\cdot \Bl -\frac{1}{2z} \frac{d}{dz} (H_\alpha+z^2)^{-1}(a,a) \Br\\
	&=\bl p\cdot W(\psi,\phi_\alpha) \br^2 \cdot \phi_\alpha(a)^{-2}
		\cdot (H_\alpha + z^2)^{-2}(a,a)\\
	&=\bl p\cdot W(\psi,\phi_\alpha) \br^2 \cdot \phi_\alpha(a)^{-2}
		\cdot \Bl -\frac{1}{2z} \frac{d}{dz} \frac{\phi_\alpha(a)\cdot \psi_z(a)}{p\cdot 
		W(\psi_z, \phi_\alpha)} \Br.
\end{split}
\end{equation}
The claim now follows by noting that $\phi_\ga(a) = p(a)^{-3/4}$ is independent of $z$
and by plugging in the known values for $\psi_z(a) = p(a)^{-1/4} \cdot \mu_\ga\ii$ and
$p\cdot W(\psi,\phi_\ga) = \gl_\ga / \mu_\ga $.
\end{proof}

Finally, let us do the following, a priori formal, calculation for
the relative $\zeta$--determinant of $H_\ga$ and $H_D$: 
\begin{equation}\label{eq.sec2.quotient}
\begin{split}
\log \frac{\detz H_\alpha}{\detz H_D} 
	&= -2 \regint_0^\infty z \Tr \Bl (H_\alpha+z^2)^{-1} - (H_D + z^2)^{-1} \Br dz \\
	&=  2 \regint_0^\infty z \Tr \Bl (H_D+z^2)^{-1} - (H_\alpha + z^2)^{-1} \Br dz \\
	&=  2 \regint_0^\infty z \int_{a}^{\infty} 
	\frac{\psi \cdot \phi}{p\cdot W(\psi_z ,\phi_z)}
       - \frac{\psi_z \cdot \phi_{\alpha,z}}{p\cdot W (\psi_z, \phi_{\alpha,z})} dz\\
	&= \regint_0^\infty \frac{2z}{p \cdot W(\psi_z, \phi_z)} 
       \Bl-\frac{\mu_{\alpha,z}}{\lambda_{\alpha,z}} \Br \int_{a}^{\infty} \psi(y)^2 dy dz\\
       &= - \regint_0^\infty \frac{d}{dz} \log \lambda_{\alpha,z} dz\\
       &=\log \lambda_{\alpha,z}|_{z=0} - \LIM_{z\to\infty} \log\gl_{\ga,z}.
\end{split}
\end{equation}

This calculation is valid whenever the trace under the first integral
has an asymptotic expansion as $z\to \infty$ as \Eqref{eq.gen.DetExp}.
Note that then the existence of the regularized limit $\LIM_{z\to\infty}
\log\gl_{\ga,z}$ follows automatically. In concrete situations,
as \textit{e.g.} in Section \ref{s.variation-model} below, the computation of this regularized 
limit (ideally proving that it is $0$) is a separate issue.

\mpar{make an appropriate reference where this argument is used below}
Instead of formulating a formal Theorem we record for later reference that
if the computations of this subsection are valid and if 
$\LIM_{z\to\infty} \log\gl_{\ga,z} = 0$ then by \Eqref{eq.Wronsk.sec2.5.1}, 
\eqref{eq.Wronski.sec2.5.2} and \eqref{eq.sec2.quotient} the quotient
$\detz H / p W(\psi,\phi)$ is \emph{independent} of the boundary condition at
$a$. Replacing $H$ by $H+\nu^2$ (\textit{e.g.} to ensure invertibility) we even conclude
from this calculation that the quotient
$\detz (H+\nu^2) / p W(\psi_\nu,\phi_\nu)$ is \emph{independent} of the boundary 
condition at $a$.

\section{Asymptotic expansions of modified Bessel functions}
\label{s.bessel}

In this section we present the relevant asymptotic expansions for
the modified Bessel functions of the first and second kind, which will
be employed throughout this paper. We employ the standard references
Abramowitz and Stegun \cite{AS}, Gradshteyn and Ryzhik \cite{GR},
Olver \cite{Olv} as well as Watson \cite{Wat}.  We will also refer to
Sidi and Hoggan \cite{Sidi} in Sec. \ref{ss.bessel-2}.

We begin by recalling the definitions of modified Bessel functions.
Modified Bessel functions of order $z \in \R$ are defined as a
fundamental system of solutions 
$f \in C^\infty (0,\infty)$ to the following
differential equation
\begin{equation}\label{Mod.Dif.Eq}
f''(x)+ \frac{1}{x} f'(x) - \left(1+\frac{z^2}{x^2}\right)\cdot f(x) = 0.
\end{equation} 
A fundamental system of solutions to this second order differential
equation is given in terms of the modified Bessel functions of first
and second kind
\begin{equation}
I_z(x): = \frac{x^z}{2^{z}} \sum_{k=0}^{\infty} \frac{x^{2k}}{4^k \cdot k!\cdot 
\Gamma(z + k + 1)}, \qquad
K_{z}(x):= \frac{(I_{-z}(x)-I_z(x))}{2\pi\sin(z \pi)},
\end{equation}
where $K_z(x)$ is defined for $z \not \in\Z$ and for $z \in \Z$
by the  limit
\begin{equation}
K_{z}(x):=\lim_{t \to z} \frac{I_{-t}(x)-I_t(x)}{2\pi\sin(t \pi)}.
\end{equation} 
We gather some important properties of the modified Bessel functions
in the following proposition. These properties are classical and can
be inferred from the aforementioned references Abramowitz and Stegun
\cite{AS}, Gradshteyn and Ryzhik \cite{GR}, as well as Olver
\cite{Olv}.

\begin{prop}\label{classical}\indent\par
\begin{enumerate}
\item The Wronskian of the fundamental system is given by $W(K_z,I_z)(x)=x^{-1}$. 
\item The modified Bessel functions are positive for $x>0$.
\item For $z$ fixed, $K_z(x)$ is decreasing and $I_z(x)$ is 
increasing.
\item For $x$ fixed and $z\in [0,\infty)$, $K_z(x)$ is increasing and $I_z(x)$ is 
decreasing.
\item $I_z\in L^1[0,a]$, but $I_{z}\not\in L^1[a,\infty)$, for $a>0$.
\item $K_z\not\in L^1(0,a]$, but $K_{z}\in L^1[a,\infty)$, for $a>0$.
\end{enumerate}
\end{prop}

We now begin with studying asymptotic expansions of the modified
Bessel functions. We distinguish between the following three cases:
large argument $x$ and fixed order $z$, fixed argument and large
order, as well as uniform asymptotic expansion for larger order.

\subsection{Asymptotics for large arguments and fixed order}
\label{ss.bessel-large}

We begin with an analysis of the asymptotic behavior of the Bessel
functions for fixed order $z \geq 0$ and large argument $x\to
\infty$.  We infer from \cite[(9.7.1), (9.7.2)]{AS}, see also
\cite[p.~202, \S 7.23 (1)-(2)]{Wat}, that the Bessel functions admit
the following asymptotic expansions 
\begin{align} 
&I_z(x) \sim \frac{e^x}{\sqrt{2\pi x}} \left( 1 + 
\sum_{k=1}^\infty (-1)^k A_k (z)\; x^{-k}\right), \quad x\to  \infty, 
\label{a.e.I.x.inf} \\
&K_z(x) \sim \sqrt{\frac{\pi}{2x}} e^{-x}  \left( 1 + 
\sum_{k=1}^\infty A_k (z)\; x^{-k}\right), \quad x\to  \infty. 
\label{a.e.K.x.inf}  
\end{align}
The coefficients of in the asymptotic expansions above are given explicitly by 
\begin{align}
A_k(z) &= \frac{1}{ 8^k k!}\prod_{n=1}^{k} 
\left(4z^2 - (2n-1)^2\right).
\end{align} 

\subsection{Asymptotics for fixed arguments and large order}
\label{ss.bessel-2}

For the asymptotics of modified Bessel functions for large order we
refer to Sidi and Hoggan \cite{Sidi}. Asymptotics of $I_z(x)$ also
follows from the asymptotic expansion of the (unmodified) Bessel
function \cite[(9.3.1)]{AS}.  Using the Stirling formula asymptotics
for the Gamma function, see \textit{e.g.}  \cite[(6.1.37)]{AS}, we infer from
\cite{Sidi} for $x>0$ fixed
\begin{align}
I_z(x) &\sim\frac{1}{\sqrt{2\pi z}} \left(\frac{ex}{2z}\right)^z
\left(1+\sum_{j=1}^{\infty}\frac{B_j(x)}{z^j}\right),  
\quad  z \to  \infty, \label{asymp.nu.I.xfixed} \\
K_z(x) &\sim\sqrt{\frac{\pi}{2z}} \left(\frac{ex}{2z}\right)^{-z}
\left(1+\sum_{j=1}^{\infty}(-1)^j\frac{B_j(x)}{z^j}\right), 
 \quad  z \to  \infty. \label{asymp.nu.K.xfixed}
\end{align}
The coefficients $B_j$ are polynomials in $(x/2)^2$ of degree $j\in \N$.
At several instances we will use a particular consequence of these
expansions.
\begin{equation}\label{KK}
  \frac{K_{z+1}(x)}{K_z(x)} = \frac{2 z}{x} + O(z\ii),\quad
  z\to\infty.
\end{equation}
Similar expansions hold for the derivatives just using the standard 
recurrence relations of Bessel functions, \cf
\cite[(9.6.26)]{AS}
\begin{equation}\label{recurrence}
I'_z(x) = I_{z+1}(x) + \frac{z}{x}I_z(x), \quad 
K'_z(x) = -K_{z+1}(x) + \frac{z}{x}K_z(x).
\end{equation}
Combining \Eqref{KK} and \eqref{recurrence} we find
\begin{equation}\label{eq.KnuprimeKnu}
   \frac{K_z'(x) }{K_z(x)} = -\frac {z}{x} +
   O(z\ii),\quad z \to\infty.
\end{equation}

\subsection{Uniform asymptotic expansion for large order}
\label{ss.bessel-3}

We now turn to uniform asymptotics of Bessel functions, when the
order go to infinity. Following Olver \cite[p. 377
(7.16), (7.17)]{Olv}, see also \cite[(9.7.7), (9.7.8)]{AS}, we have
for large $z>0$ and uniformly in $x > 0$ 
\begin{equation}\label{a.e.I.nu.inf}
I_z(z x) \sim \frac{e^{z \xi(x)}}{\sqrt{2\pi z} 
(1+x^2)^{\frac{1}{4}}}\left(1+\sum_{k=1}^{\infty}  
\frac{U_k(x)}{z^k} \right), \quad z \to \infty.
\end{equation}
Similarly, for the modified Bessel functions of second kind we have
for large $z>0$ and uniformly in $x > 0$ the following asymptotic
expansions
\begin{equation}\label{a.e.K.nu.inf}
K_z(z x) \sim\sqrt{\frac{2\pi}{z}} \frac{e^{-z \xi(x)}}{(1+x^2)^{\frac{1}{4}}} 
\left(1+\sum_{k=1}^{\infty} (-1)^k\frac{U_k(x)}{z^k} \right), \quad z \to \infty.
\end{equation} 
In both cases we have introduced the following notation
\begin{equation}\label{pnu}
  \begin{split}
    \xi     & = \xi(x) := \sqrt{1+x^2} + \log \frac{x}{1+\sqrt{1+x^2}}, \\
    {\rm p} & = {\rm p}(x) := \frac{1}{\sqrt{1+x^2}}.
  \end{split}
\end{equation}
The coefficients $U_k(x)$ in the asymptotic expansions above, are polynomial functions in 
${\rm p}$ of degree $3k$. 

We conclude with an asymptotic expansion for a product of Bessel
functions. Using Cauchy product formula we infer from the expansions
above
\begin{equation}\label{uniform1}
I_z(z x) K_z(z x)  =\frac{1}{2z} \frac{1}{\sqrt{1+x^2}} \left(1+ \sum_{k=1}^\infty 
\frac{\tilde U_{2k}(x)}{z^{2k}}\right),
\quad z \to \infty,
\end{equation}
where the coefficients $\tilde U_{2k}(x)$ are polynomials in ${\rm p}$ of
degree $6k$.

\section{Variation formula and the determinant of the model operator}
\label{s.variation-model}

Fix $\mu>0$ and consider the family of scalar model \emph{cusp} operators 
\begin{equation}\label{eq.cusp.op}
  D_\mu  = -\frac{d}{dx}\Bigl(x^2 \frac{d}{dx}\cdot\Bigr) + x^2\mu^2 -\frac 14: 
            C_0^\infty(a,\infty) \to  C_0^\infty(a,\infty).
\end{equation}
Let $z\geq 0$. Then a fundamental system of solutions for the second order
differential equation $(D_\mu + z^2) f = 0$ is given in terms of the modified
Bessel functions by $x^{-1/2}I_z(\mu x), x^{-1/2}K_z(\mu x)$. By
\Eqref{a.e.I.x.inf}, $x^{-1/2}I_z(\mu x)$ does not lie in $L^2[a,\infty)$.
Consequently, $\infty$ is in the limit point case for the operator $D_\mu + z^2$
and self-adjoint extensions are obtained by imposing boundary
conditions of the form \Eqref{INTRO.eq.2} at $x=a$.

Consider first the case of Dirichlet boundary conditions $R_af=f(a)$. By abuse of 
notation we use $D_\mu := D_\mu (R_a)$.  A
normalized fundamental system, \cf \Eqref{INTRO.eq.3} and Theorem
\ref{T.main}, of solutions to the equation $(D_\mu + z^2) f=0$ is then given
by
\begin{equation}\label{eq.FS-model-Dir}
\psi_{z,\mu}(x) = x^{-1/2} K_z(\mu x), \quad
\phi_{z,\mu}(x) = x^{-1/2} \bl K_z(\mu a) \cdot I_z(\mu x) - I_z(\mu a) \cdot K_z(\mu x)\br.
\end{equation} 
Note that $\psi_{z,\mu}\in L^2[a,\infty)$. Furthermore,
\begin{equation}
    \phi_{z,\mu}'(a)  = a^{-1/2} \cdot \mu\cdot W(K_z, I_z)(\mu a) = a^{-3/2}
    =   p(a)^{-3/4},
\end{equation}  
hence $\phi_{z,\mu}$ has the correct normalization according to
\Eqref{INTRO.eq.3}. The Wronskian of the modified Bessel functions satisfies
$W(K_z, I_z)(x) = \frac 1x$. Furthermore, we point out that $K_z(\mu x) > 0$
is nowhere vanishing for $x>0$ by Proposition \ref{classical}. In particular,
$D_\mu(R_a)+z^2$ is invertible.  The Wronskian of $\psi_{z,\mu},\phi_{z,\mu}$ is
given by 
\begin{equation}
  x^2\cdot W(\psi_{z,\mu}, \phi_{z,\mu}) = x \cdot\mu\cdot K_z(\mu a)\cdot W(K_z,
  I_z)(\mu x) = K_z(\mu a).
\end{equation}
Consequently, the Green function $G_z$ of $(D_\mu + z^2)\ii$ is obtained as in
\Eqref{eq.gen.green-kernel}
\begin{equation}\label{GreenFunction}
G_z(x,y) = 
  \begin{cases}
    (x y)^{-1/2} \cdot\bl I_z(\mu x)\cdot K_z(\mu y) - \frac{I_z(\mu a)}{K_z(\mu a)} K_z(\mu x)\cdot 
                 K_z(\mu y)\br,  &  x\leq y,  \\
    (x y)^{-1/2}\cdot\bl I_z(\mu y)\cdot K_z(\mu x) - \frac{I_z(\mu a)}{K_z(\mu a)} K_z(\mu y)\cdot 
                 K_z(\mu x)\br,  &  y\leq x.
   \end{cases}
\end{equation}
In particular we find for the Green function at the diagonal
\begin{equation}\label{resolvent-kernel}
G_z(x) \equiv G_z(x,x) = x^{-1} 
\bl I_z(\mu x)\cdot K_z(\mu x) - \frac{I_z(\mu a)}{K_z(\mu a)} K^2_z(\mu x) \br.
\end{equation}
The Green function $G_z(x,y)$ is continuous on $[a, \infty)^2$ and by positivity
of the solutions $\phi_{z,\mu}$ and $\psi_{z,\mu}$, the kernel is non-negative and positive 
away from $x, y = a$.  Moreover, $G_z(x) = O(x^{-2})$ as $x \to \infty$ by the 
asymptotic expansions \Eqref{a.e.I.x.inf} and \eqref{a.e.K.x.inf}. Hence $G_z$ is 
integrable 
on $[a,\infty)$ along the diagonal.
Consequently, by Mercer's theorem, as worked out \eg by Reed and Simon
\cite[\S XI.4, Lemma on p. 65]{Reed} we conclude that the resolvent 
$(D_\mu + z^2)^{-1}$ is trace class and the trace of $(D_\mu + z^2)^{-1}$ 
is given by  
\begin{equation}
 \Tr\bl D_\mu + z^2\br^{-1} = \int_a^\infty G_z(x) dx.
\end{equation}
A similar argument holds in case of generalized Neumann boundary conditions
$R_a = f'(a)+\alpha f(a)$. In that case the solution $\phi_{z,\mu,\ga}$, satisfying
the generalized Neumann boundary conditions, is given by 
\begin{equation}\label{eq.FS-model-Neu}
  \phi_{z,\mu,\ga}(x) = c_{z,\mu}\cdot x^{-\frac12}\cdot \left(
     I_z(\mu x) -  \frac{%
       \bl\alpha-\frac{1}{2 a}\br\cdot I_z(\mu a)+ \mu \cdot  I'_z(\mu a)}{%
     \bl \alpha-\frac{1}{2 a}\br \cdot K_z(\mu a)+\mu  \cdot K'_z(\mu a)}
   \cdot K_z(\mu x)
   \right).
\end{equation} 
The constant $c_{z,\mu}$ is determined by the normalization requirement
\Eqref{INTRO.eq.3}, \ie $\phi_{z,\mu}(a) = p(a)^{-1/4} = a^{-1/2}$.  We construct
the Green function as before.  But now $G_z$ is positive only if either
$\alpha \in \left[0,\frac{1}{2}\right]$ or if $\alpha>\frac{1}{2}$ and $\mu$
sufficiently large. In these two cases we can use Mercer's Theorem as before
to obtain $\Tr(D_\mu (R_a) + z^2)^{-1}$ integrating the Green function $G_z$ along the 
diagonal.
If $G_z$ is not necessarily positive, the trace class property 
of $(D_\mu(R_a) + z^2)^{-1}$ follows by a well--known comparison 
principle for elliptic operators, \cf \textit{e.g.,} \cite[Chapter 3]{LMP:CCC}
and \cite{LesVer}. 

\begin{prop}\label{p:Comparison}
Consider Dirichlet boundary conditions $R_{a,1}(f)=f(a)$ and generalized Neumann
boundary conditions $R_{a,2}(f)=f'(a)+\alpha f(a)$ for the model operator $D_\mu$.
Let $D_{\mu, 1}$ and $D_{\mu, 2}$ denote the corresponding self-adjoint
realizations in $L^2[a,\infty)$ with boundary conditions $R_{a,1}$ and $R_{a,2}$,
respectively. Then
\begin{equation}
\bigl\| (D_{\mu, 1} + z^2)^{-1} - (D_{\mu, 2} + z^2)^{-1}\bigr\|_{\tr} = O(z^{-2} \log 
z), 
\quad \text{as} \quad z \to \infty.
\end{equation}
\end{prop}

\begin{proof}
Fix any $\delta > \delta' > a$. 
We choose cut-off functions $\chi_1, \chi_2\in C^{\infty}_0[a,\delta)$, 
as illustrated in Figure \ref{fig:CutOff}, 
such that they are identically one over $[a,\delta']$ and moreover, 
\begin{itemize}
\item $\supp(\chi_1)\subset \supp(\chi_2)$,
\item $\supp(\chi_1)\cap \supp(d\chi_2)=\emptyset$.
\end{itemize}

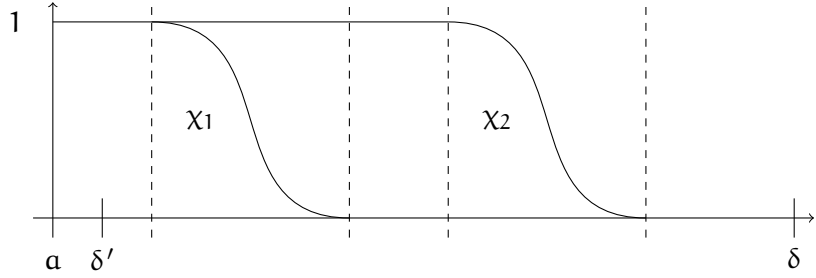
\begin{figure}[h]
\begin{center}

\begin{tikzpicture}[scale=1.3]
\draw[->] (-0.2,0) -- (7.7,0);
\draw[->] (0,-0.2) -- (0,2.2);

\draw (-0.2,2) node[anchor=east] {$1$};
\draw (0.5,-0.2) node[anchor=north] {$\delta'$} -- (0.5,0.2);
\draw (7.5,-0.2) node[anchor=north] {$\delta$} -- (7.5,0.2);
\draw (0,-0.45) node {$a$};
\draw (0,2) -- (4,2);
\draw (1,2) .. controls (2.4,2) and (1.6,0) .. (3,0);
\draw[dashed] (1,-0.2) -- (1,2.2);
\draw[dashed] (3,-0.2) -- (3,2.2);
\draw (4,2) .. controls (5.4,2) and (4.6,0) .. (6,0);
\draw[dashed] (4,-0.2) -- (4,2.2);
\draw[dashed] (6,-0.2) -- (6,2.2);

\draw (1.5,1) node {$\chi_1$};
\draw (4.5,1) node {$\chi_2$};

\end{tikzpicture}

\caption{The cutoff functions $\chi_1$ and $\chi_2$.}
\label{fig:CutOff}
\end{center}
\end{figure}

We write $\eta_1:=1-\chi_1$, $\eta_2:= 1- \chi_2$. 
By construction $\eta_1\eta_2=\eta_2$. We consider 
\begin{equation}
R(z):=\eta_1 \bigl[(D_{\mu, 1}+z^2)^{-1}-(D_{\mu, 2}+z^2)^{-1}\bigr]\eta_2 .
\end{equation}
$R(z)$ maps into the domain of both $D_{\mu, 1}$ and $D_{\mu, 2}$. 
On the support of $\eta_1 $ the differential expressions $D_{\mu, 1}$ 
and $D_{\mu, 2}$ coincide and moreover $\eta_1  \dom(D_{\mu, 1})=\eta_1 \dom(D_{\mu, 
2})$. 
Thus we may compute 
\begin{align}
(D_{\mu, 1}+z^2)R(z)=[D_{\mu, 1},\eta_1 ]\bigl((D_{\mu, 1}+z^2)^{-1}-(D_{\mu, 
2}+z^2)^{-1}\bigr).
\end{align}
Arguing similarly for $R(z)^*$ and taking adjoints one then finds
\begin{equation}
(D_{\mu, 1}+z^2)R(z)(D_{\mu, 2}+z^2)=[-\partial_x^2, \eta_1 ]
\bigl((D_{\mu, 1}+z^2)^{-1}-(D_{\mu, 2}+z^2)^{-1}\bigr)[\partial_x^2,\eta_2 ],
\end{equation}
where $[\cdot, \cdot]$ denotes the commutator of the corresponding operators
and any function is viewed as a multiplication operator.  Hence
\begin{equation}
R(z)=(D_{\mu, 1}+z^2)^{-1}[-\partial_x^2,\eta_1 ]
\bigl((D_{\mu, 1}+z^2)^{-1}-(D_{\mu, 2}+z^2)^{-1}\bigr) 
[\partial_x^2,\eta_2 ](D_{\mu, 2}+z^2)^{-1}.
\end{equation}
Since $(D_{\mu, 1}+z^2)^{-1}$ is trace class by Mercer's theorem as explained
above, and since the space of trace class operators forms an ideal in the
space of bounded operators, we conclude that $R(z)$ is trace class as well and
continue with the following estimate
\begin{equation}
\begin{split}
\|R(z)\|_{\tr}\leq \|(D_{\mu, 1} + z^2)^{-1}\|_{\tr}\Bigl(\|[\partial_x^2,\eta_1 ](D_{\mu, 
1}+z^2)^{-1}\|+\|[\partial_x^2,\eta_1 ](D_{\mu, 2}+z^2)^{-1}\|\Bigr)\cdot\\
                         \cdot \|[\partial_x^2,\eta_2 ](D_{\mu, 2}+z^2)^{-1}\|.
\end{split}
\end{equation} 
By \Eqref{eq.INTRO.resolvent.model}, we have 
$\|(D_{\mu, 1}+z^2)^{-1}\|_{\tr}=O(z^{-1}\log z)$. Let $f$ denote $\eta_1 $ or
$\eta_2 $. Then $[\partial_x^2, f]$ is a first order differential operator
whose coefficients are compactly supported in $(a,\delta)$, hence it maps the
Sobolev space $H^1[a,\delta]$ continuously into $L^2_{\comp}(a,\delta)$.
Therefore we conclude for $j=1,2$ with a cut--off function $\chi \in
C^{\infty}_0(a,\delta)$ with $\chi=1$ in a neighborhood of 
$\supp ([\pl_x^2,f])$,
\begin{equation}
\|[\partial_x^2,f](D_{\mu, j}+z^2)^{-1}\|\le \|[\pl_x^2,f]\|_{H^1\to L^2}
\|\chi (D_{\mu, j}+z^2)\ii\|_{L^2\to H^1} =O(z^{-1}),
\end{equation}
as $z\to \infty$. Hence for $j=1,2$ the operator norms 
$\|[\partial_x^2,\eta_1 ](D_{\mu, j}+z^2)^{-1}\|$ and 
$\|[\partial_x^2,\eta_2 ](D_{\mu, j}+z^2)^{-1}\|$ behave 
as $O(z^{-1})$ as $z\to \infty$. This proves  
\begin{equation}
\bigl\| \left.\left((D_{\mu, 1} + z^2)^{-1} - (D_{\mu, 2} + z^2)^{-1}
\right)\right|_{L^2(\delta, \infty)}\bigr\|_{\tr} = O(z^{-2}\log z), 
\quad \text{as}\quad z \to \infty.
\end{equation}
Note that the Schwartz integral kernel of $(D_{\mu, 2} + z^2)^{-1}$ is smooth
at the diagonal $[a,\infty)$ and hence by continuity is strictly positive (or
strictly negative) over $[a,\delta]$ for $(\delta-a)>0$ sufficiently small.
By Mercer's theorem, $(D_{\mu, 2} + z^2)^{-1}$ is trace class in
$L^2(a,\delta)$.

The statement now follows from the fact that integrals of the Schwartz kernels
for $(D_{\mu, 1} + z^2)^{-1}$ and $(D_{\mu, 2} + z^2)^{-1}$ along the diagonal
in $[1,\delta]$ admit an asymptotic expansion of the form $\sum_{k=0}^{\infty}
a_k(\mu) z^{-1-k}$, where the leading order term $a_0$ is independent of the
boundary conditions.  
\end{proof}

The next proposition is proved in \cite{Ver} without specifying the leading
coefficients in the asymptotic expansion. By keeping track of the coefficients
in the asymptotic expansions of Sec.
\ref{ss.bessel-3} we obtain the following more precise statement.

\begin{prop}\label{Asympt.Dt} For any boundary condition of the form
\Eqref{INTRO.eq.2} at $a$ the corresponding self-adjoint realization
$D_\mu(R_a)$ of $D_\mu$ admits  an asymptotic expansion of the resolvent
trace
\begin{equation}\label{eq.resolvent}
  \Tr (D_\mu(R_a) + z^2)^{-1} \sim \sum_{k=0}^{\infty} a_k(\mu)\cdot z^{-1-k}
+ \sum_{k=0}^{\infty} b_{2k}(\mu) \cdot z^{-1-2k}\cdot\log z,
\quad \textup{as} \quad z\to\infty,
\end{equation} 
where $a_0(\mu) =\frac{1}{2}\log \frac{2}{\mu a}$ and $b_{0}(\mu) =\frac{1}{2}$
independent of the choice of boundary conditions at $a$. Moreover, for
Dirichlet boundary conditions $a_1(\mu) = \frac{1}{4}$, while for generalized
Neumann boundary conditions $a_1(\mu) = -\frac{1}{4}$.
\end{prop}

Note that the asymptotic expansion \Eqref{eq.resolvent} does not admit terms of the form
$z^{-2}\log^k(z), k\in \N$. The zeta-regularized determinant of $D_\mu(R_a)+\nu^2$, for 
any $\nu \geq 0$, is therefore defined according to Sec. \ref{ss.gen.ZetaDeterminant}.

The following variation formula is due to the third named author
\cite[Proposition 5.1 and Theorem 5.2]{Ver}. We present it here with precise
formulae for solutions and the normalizing constants.
Theorem \ref{W-variation-thm}, for which we will give a complete proof,
contains the following as a special case.

\begin{theorem}\label{Theo.Vert} Fix a boundary condition $R_a$ for the model
operator $D_\mu$ at $a$. For $\nu\geq 0$, let 
$\psi_{\nu,\mu} (x) = x^{-\frac{1}{2}}K_\nu(\mu x)$ and $\phi_{\nu,\mu}(x) $
be a normalized fundamental system of  solutions to $(D_\mu + \nu^2) f = 0$,
\cf \Eqref{INTRO.eq.2}, \eqref{eq.FS-model-Dir}, \eqref{eq.FS-model-Neu}.

\mpar{$c_{\nu,\mu}$ is explicit, do it}
Assume that the null space of $(D_\mu + \nu^2)$ is trivial. Then the
zeta-regularized determinant of $(D_\mu + \nu^2)$ is differentiable in $\nu$
and satisfies the following variational formula
\begin{equation}\label{Var.form.}
\frac{d}{d\nu} \log\detz (D_\mu + \nu^2) = \frac{d}{d\nu} \log \bl x^2\cdot  
W(\psi_{\nu,\mu}(x),\phi_{\nu,\mu}(x))\br.
\end{equation}
\end{theorem}

Now we are already in a position to compute the zeta-regularized determinant
for the model operator $D_\mu$ for any boundary condtion,
\cf \eg \cite[Thm.~3.3]{LesTol:DOD}.
\begin{theorem} \label{T.ModelDetFormula} Fix a boundary condition
  for the model operator $D_\mu$. Then for the zeta-regularized determinant
  of $D_\mu(R_a)+\nu^2$ we have
\begin{equation}
  \detz (D_\mu(R_a) +\nu^2 )= \sqrt{\frac 2\pi } \cdot a^2\cdot W(\psi_\nu, 
  \phi_{\nu,R_a})(a).
\end{equation}  
Here, $\phi_{\nu,R_a}, \psi_\nu$ is a normalized fundamental system of solutions
to the equation $(D_\mu+\nu^2) f = 0$, $\psi_\nu(x) = x^{-1/2} K_\nu(\mu x)$ and
$\phi_{\nu, R_a}$ is given in \Eqref{eq.FS-model-Dir} (Dirichlet) resp.
\Eqref{eq.FS-model-Neu} (Neumann).
\end{theorem}  

\begin{proof}
From the previous Theorem and Lemma \ref{L.gen.asymptotic-det} we infer
\begin{equation}
  \detz (D_\mu(R_a)+\nu^2) = a^2\cdot W(\psi_{\nu},\phi_{\nu,R_a}) \cdot
  \exp\Bl -\LIM_{z\to\infty} \log\bl a^2\cdot W(\psi_z, \phi_{z,R_a}) \br\Br.
\end{equation}     
It therefore remains to compute the $\LIM$ in the exponential function
on the right. Let us first look at the case of Dirichlet boundary conditions.
Let $\phi_z = \phi_{z,R_a}$ for $R_a f = f(a)$. Then
\begin{equation}
a^2 \cdot W(\psi_{z},\phi_{z})(a) = a^2 \cdot \psi_{z} (a)\cdot 
	\phi_{z}'(a) = K_z(\mu  a),
\end{equation}
since $\phi_z$ is normalized to $\phi_z'(a) = a^{-3/2}$.
Using the asymptotic expansion \Eqref{asymp.nu.K.xfixed}, we obtain
\begin{equation}
  \log  K_z(\mu a) = \log \sqrt{\frac{\pi}{2}} - 
	\frac{1}{2} \log z - z \log \Bl \frac{e \mu a}{2z } \Br + \log \Bl 1 + 
	O(z^{-1}) \Br,\;{\text{as}}\;z\to\infty,
\end{equation} 
hence $\LIM\limits_{z\to\infty} \log \bl K_z(\mu a) \br = \log \sqrt{\frac \pi 2}$
and the result follows.

Next consider a generalized Neumann boundary condition $R_a f = f'(a) + \ga
f(a)$ and denote by $\phi_{z,\ga}$ the corresponding normalized solution
satisfying $R_a \phi_{z,\ga} = 0$. Then by \Eqref{eq.DirNeu1},
\eqref{eq.DirNeu2} 
\begin{equation}
	\begin{split}
  	\frac{a^2 \cdot W(\psi_{z},\phi_{z,\alpha})}{a^2\cdot W(\psi_{z},\phi_{z})} 
  	& = \lambda_\alpha = - a \cdot \frac{R_a\psi_{z}}{\psi_{z}(a)}
	   = - a^{\frac 12} \cdot \frac{(\alpha - \frac{1}{2a})~K_z(\mu a) + \mu~K'_z(\mu 
		a)}{a^{-\frac 12}K_z(\mu a)}\\
	  & =-\bl \alpha a - \frac 12\br - \mu a \cdot \frac{K'_z(\mu a)}{K_z(\mu
  a)} \\
    & = z + O(1) = z \cdot \bl 1+ O(z\ii)\br,\quad\text{ as } z\to \infty.
	\end{split}
\end{equation}
The last line follows from \Eqref{eq.KnuprimeKnu}. Taking log on both sides
yields
\begin{equation}
	\log \Bl\frac{a^2 \cdot W(\psi_{z},\phi_{z,\alpha})}{%
    a^2\cdot W(\psi_{z},\phi_{z})} \Br 
    =	\log z + 	O(z^{-1}),\;\text{ as } z\to \infty,
	\end{equation}
consequently the regularized limit for the Neumann boundary condition
equals that for the Dirichlet boundary condition, \ie
\begin{equation*}
	\LIM_{z\to \infty} \log \bl a^2 \cdot W(\psi_{z},\phi_{z,\alpha})
  \br
  =\LIM_{z\to \infty}	\log\bl a^2\cdot W(\psi_{z},\phi_{z}) \br
  =\log\sqrt{\frac\pi2}.\qedhere
\end{equation*}
\end{proof}

\section{B{\^o}cher theorem for operators with quadratic potentials at infinity}
\label{Bocher}

In this section we will prove a version of the B\^ocher's Theorem for
perturbations of the model cusp operator \Eqref{eq.cusp.op} and we analyse the 
Wronskian's behavior at infinity of a perturbed fundamental system of solution. 

Recall that
$\psi_z(x) = x^{-\frac{1}{2}}K_z(\mu x)$, $\phi_z(x) = x^{-\frac{1}{2}}I_z(\mu x)$
is a fundamental system of solutions to the equation $(D_\mu +z^2) f = 0$ 
with Wronskian
\begin{equation}
x^2\cdot W\bl \psi_z, \phi_z \br(x) = 1.
\end{equation} 
In the notation of Section \ref{s.generalities} we have $p(x) = x^2$.
We will specialize the result of Section \ref{ss.Pert-Boch} to
\Eqref{eq.cusp.op}. For this we need to verify the conditions \Eqref{Assump1}
- \eqref{Assump5}.

\begin{lemma}\label{lemmaLt}
For the operator \Eqref{eq.cusp.op} we have, \cf \Eqref{kernelL},
\begin{equation}\label{VolterraKernel}
		L(x,y) := y\cdot  K_z^2(\mu y) \left[\frac{I_z(\mu y)}{K_z(\mu y)} - 
		\frac{I_z(\mu x)}{K_z(\mu x)}\right],
	\end{equation} 
for $a\leq  x \leq y<\infty$. Furthermore,
\begin{equation}
	\sup_{a\leq x\leq y<\infty}| L (x,y)| \leq C(\mu).
\end{equation}
\end{lemma}
\begin{proof} This follows from the asymptotic expansion \Eqref{a.e.I.x.inf}
and \eqref{a.e.K.x.inf}. Namely, choose $y_0$ such that for $y\geq y_0$,
\begin{equation}	
	K_z(\mu y) \leq 2 \cdot \sqrt{\frac{\pi}{2\mu y}} e^{-\mu y},\qquad  I_z(\mu y) \leq 
	2 \cdot \frac{1}{\sqrt{2\pi\mu y}} e^{\mu y}.
\end{equation}
	Since $x\mapsto \frac{I_z(\mu x)}{K_z(\mu x)}$ is an increasing function we then have 
	for all $a\leq x \leq y$ and $y\geq y_0$
	\begin{equation}
	\bigl|L(x,y)\bigr| = y K_z^2(\mu y) \left[\frac{I_z(\mu y)}{K_z(\mu y)} - 
	\frac{I_z(\mu x)}{K_z(\mu x)}\right]
		\leq y K_z(\mu y) I_z(\mu y) 
			\leq \frac{2}{\mu}.
	\end{equation}
	Since $L$ is certainly continuous, it is bounded on the compact set $a\leq x \leq y 
	\leq y_0$ and the claim follows.
\end{proof}

\begin{theorem}\label{T.Wronsk} Let
\begin{equation}
		H = D_\mu  + X^{2}\cdot W,
\end{equation}
with $W\in  L^1[a,\infty)$ and fix $z\geq 0$. Then the differential 
equation $(H+z^2)f=0$ has a fundamental system of solutions $h_1,h_2$, such that
\begin{equation}\label{eq.h12}
			h_1(x) = \psi_z(x)\cdot g_1(x),\qquad	h_2(x) =\phi_z(x) \cdot g_2(x),
\end{equation} 
with $ g_j \in C_b[a,\infty)$ and 
$\lim\limits_{x\to\infty}g_j(x) =1$, $j=1,2$. 
Furthermore,
\begin{equation}\label{eq.h'12}
			h'_1(x) = \psi'_z(x) \cdot \tilde{g}_1(x), \qquad 
      h'_2(x) = \phi'_z(x) \cdot \tilde{g}_2(x),
\end{equation} 
where $\tilde{g}_j \in C_b[a,\infty)$, 
$\lim\limits_{x\to\infty}\tilde{g}_j(x)  =1$, $j=1,2$, 
and $x^2\cdot W(h_1,h_2)(x) = 1$.
\end{theorem}

\begin{proof}
We just need to verify the assumptions \Eqref{Assump2} - \eqref{Assump5}.
In view of the asymptotic expansion \Eqref{a.e.K.x.inf} it is easy to see
that $\psi$ satisfies \Eqref{Assump2} - \eqref{Assump4}. Assumption
\Eqref{Assump5} follows from the  asymptotic expansion \Eqref{a.e.I.x.inf} 
and \eqref{a.e.K.x.inf}. Namely, by these expansions
\begin{equation}	
\frac{I_z(\mu x)}{K_z(\mu x)} = O \bl e^{2\mu x}\br,\quad x\to \infty,
\end{equation} 
thus the assumption follows observing that quotient 
$\frac{\phi_z(x)}{\psi_z(x)}$ is $\frac{I_z(\mu x)}{K_z(\mu x)}$.
\end{proof}

\subsection{Asymptotics of Wronskians for the perturbed operator}
\label{s.asymptotics_Wronskians}

Consider a family of functions $W_t\in  L^{1}[a,\infty)$, depending on a 
real parameter $t$ such that $t \mapsto W_t$ is differentiable as a map 
into $L^{1}[a,\infty)$. We apply Theorem \ref{T.Wronsk} to study 
the fundamental system of solutions and their Wronskians for the 
perturbed operator
\begin{equation}\label{Pert.Op}
H_t+z^2:=(D_\mu +z^2) + X^{2}W_t.
\end{equation} 
As a notational convenience we take $t_0 = 0$ as base point.
By Theorem \ref{T.Wronsk} the equation $(H_t + z^2) f = 0$ has two solutions. Let 
$h_{1,t}(x)$ be the solution 
of $(H_t + z^2) h_{1,t} = 0$ with 
\begin{equation}\label{sec5.norm.infinity}
h_{1,t}(x) \sim  x^{-\frac{1}{2}}K_z(\mu x) = O(x^{-1}e^{-\mu x}),\quad 
x\to\infty.
\end{equation}
Note that by \Eqref{sec5.norm.infinity} the solution $h_{1,t}$ is unique as the space of 
solutions in the limit point case has dimension one.

\begin{lemma}	The solution $h_{1,t}(x)$ is differentiable in $t$ with the
estimates $\dot h_{1,t}(x) = o(x^{-1}e^{-\mu x})$, 
$\b_x \dot{h}_{1,t}(x) = o(x^{-1}e^{-\mu x})$ as 
$x\to\infty$ and the $o$-constants are locally uniform in $t$, \ie 
$h_{1,t}(x) - h_{1,0}(x) = o(x^{-1}e^{-\mu x})$ as $x\to\infty$.
\end{lemma}

\begin{proof} The $o$-behavior as $x\to\infty$ follows directly from
Lemma \ref{L.op}. In fact, 
\begin{equation}
	\begin{split}
    	       \dot h_{1,t} &= \psi_z \cdot \dot f_{1,t}\\
	     \b_x \dot{h}_{1,t} &= \psi'_z \cdot \dot f_{1,t} + \psi \cdot \b_x \dot{f}_1,
	\end{split} 
\end{equation} 
where by equation \Eqref{f1.LW},
\begin{equation}
	\dot f_{1,t} = (I-LW_t)^{-1}(L\dot{W}_t {\bf 1})
	+ (I-LW_t)^{-1}L\dot{W}_t(I-LW_t)^{-1}({\bf 1}),
\end{equation} 
and
\begin{equation}
	\b_x \dot{f}_{1,t} = (I-LW_t)^{-1} [ L \dot{W}_t {\bf 1}]'
	+ (I-LW_t)^{-1} [ L\dot{W}_t f_{1,t}]'.
\end{equation}
Then $ \dot{f}_{1,t}$ and $\b_x \dot{f}_{1,t} $ are in 
$C_\bullet[a,\infty)$.
	
The last part follows as the $o$-constants are locally independent of $t$ and
\begin{equation}
	 \dot{h}_{1,0}(x) = \lim_{t\to 0} \frac{h_{1,t}(x) - h_{1,0}(x)}{t}.
   \qedhere
\end{equation}
\end{proof}

Now we will choose the second solution $h_{2,t}(x)$. By equation \Eqref{f1.LW} there 
exists $x_0\in [a,\infty)$ such that $h_{1,0}(x)\not = 0$ for $x\geq x_0$. Then for 
$t$ in a neighborhood of $0$ we have $h_{1,t}(x)\not = 0$ for all $x\geq x_0$.
Thus
\begin{equation}
\begin{split}
   (xh_{1,t}(x))^{-2} - (xh_{1,0}(x))^{-2} &= \frac{(h_{1,0}(x) + 
    h_{1,t}(x))(h_{1,0}(x)-h_{1,t}(x))}{x^2 h_{1,0}^2(x)h_{1,t}^2(x)}\\
                 &= o(e^{2\mu x}), \qquad x\to\infty.
\end{split} 
\end{equation}
Define 
\begin{align}
h_{2,0}(x) = h_{1,0}(x)\int_{x_0}^{x} (yh_{1,0}(y))^{-2} dy,
\end{align} 
and 
\begin{align}
h_{2,t}(x) = h_{1,t}(x) \int_{x_0}^x (yh_{1,t}(y))^{-2} - (yh_{1,0}(y))^{-2}dy - 
\frac{h_{1,t}(x)}{h_{1,0}(x)} h_{2,0}(x).
\end{align} 

Note that, if $f(x) = o(e^{cx})$, with $c>0$ then for $x\to\infty$
\begin{equation}
\int_a^x f(x) dx = o(e^{cx}),
\end{equation} as well.

\begin{lemma}	Let $h_{1,t}, h_{2,t}$ be a fundamental system of solutions for \Eqref{Pert.Op}, 
which satisfy \Eqref{eq.h12} and \eqref{eq.h'12}. Then $h_{2,t}$ is differentiable 
in $t$ and we have
\begin{equation}
	\begin{split}
   	h_{2,t}(x) &= h_{2,0}(x)+O(x^{-1}e^{\mu x}) = O(x^{-1}e^{\mu x}),\\
	    \dot{h}_{2,t}(x)&= o(x^{-1}e^{\mu x}),\\
	    \b_x\dot{h}_{2,t}(x) &=o(x^{-1}e^{\mu x}).
	\end{split}
\end{equation}
\end{lemma} 

\begin{proof} Since
\begin{equation}
	(x\ h_{1,t}(x))^{-2} - (x\ h_{1,0}(x))^{-2} = o(e^{2x}),
\end{equation} 
the integral 
\begin{equation}
	\int_{R_0}^x (y\ h_{1,t}(y))^{-2} - (y\ h_{1,0}(y))^{-2}dy = o(e^{2x}).
\end{equation} 
This implies that $h_{2,t}(x) = O(x^{-1}e^{\mu x})$. The orders of 
$\dot{h}_{2,t}(x)$ and $\b_x \dot{h}_{2,t}(x)$ follows using last result and
the previous lemma.
\end{proof}

Now we are ready to estimate the behavior of the following Wronskians
as $x\to\infty$:
\begin{cor}\label{Wron.Corol.3} Let $h_{1,t}, h_{2,t}$ be a fundamental system of 
	solutions for \Eqref{Pert.Op}, which satisfy \Eqref{eq.h12} 
and \eqref{eq.h'12}. Then as $x\to\infty$,
\begin{equation}
	\begin{split}
	x^2 \cdot W(h_{1,t}, \dot{h}_{1,t})(x) &=  o(e^{-2\mu x});\\
	x^2 \cdot W(h_{2,t}, \dot{h}_{1,t})(x) &=  o(1);\\
	x^2 \cdot W(h_{1,t}, \dot{h}_{2,t})(x) &=  o(1);\\
	x^2 \cdot W(h_{2,t}, \dot{h}_{2,t})(x) &=  o(e^{2\mu x}).
	\end{split}
\end{equation} 
In particular, 
\begin{equation}
	\lim_{x\to\infty} x^2W(h_{1,t}, \dot{h}_{1,t})(x)= \lim_{x\to\infty} 
	x^2W(h_{2,t},\dot{h}_{1,t})(x)=0.
\end{equation}
\end{cor}
\section{Regularized determinant of the perturbed operator}
\label{s.regdet_perturbed}

In this section we establish a partial asymptotic expansion for the resolvent
trace of the perturbed operator $H+\nu^2$, which allows the definition of its 
zeta-regularized determinant.  In the case of the model operator we have the full
asymptotic expansion of the trace of $(D_\mu + z^2)^{-1}$ when $z\to\infty$. 
The Green function \Eqref{GreenFunction} and 
the uniform asymptotic expansion of the modified Bessel function \Eqref{a.e.I.nu.inf} and 
\eqref{a.e.K.nu.inf} are the main ingredients to obtain that result. This argument does 
not apply to the perturbed case, however using a Neumann series argument we can still 
derive a partial asymptotic expansion for the resolvent trace.

The results in this section are independent of the boundary
conditions at $x=a$, hence for simplicity we use Dirichlet boundary
conditions $R_a f = 0$ and by abuse of notation $D_\mu:=D_\mu (R_a)$.

\mpar{TODO: explain better, refer to corresponding argument for 
model operator, see Sec \ref{s.variation-model}}

\begin{lemma}\label{L.regdet.1}
For fixed $z$, $\mu$ and real numbers $\alpha$, $\beta$ with 
$\alpha+\beta \leq 2 $ the operator 
$X^{\alpha}(D_\mu+z^2)^{-1}X^{\beta}$  is a bounded operator in the Hilbert
space $L^{2}[a,\infty)$.
\end{lemma}

\begin{proof} We apply Schur's test \cite[Thm.~5.2]{HaSu} to the kernel
function of the operator. Recall from Section \ref{s.variation-model} that the
kernel of $(D_\mu +z^2)^{-1}$ is given by $G_z(x,y)$ \Eqref{GreenFunction}.
During the proof $C$ denotes a generic constant depending on $z$, $\mu$, $\ga$, $\gb$,
$\gamma$, $a$ but not on $x, y$; it may change from line to line.

From \Eqref{a.e.I.x.inf} and \eqref{a.e.K.x.inf} we conclude that
\begin{equation}
	\bigl|\psi_z(x)\bigr|\leq C \cdot \frac{e^{-\mu x}}{x} ,\qquad
	\bigl|\phi_z(x)\bigr|\leq C \cdot \frac{e^{\mu x}}{x}.
\end{equation}
Furthermore, we need the inequalities 
\begin{equation}
	\int_{a}^{x} e^{\mu y} y^\gamma dy \leq C\cdot e^{\mu x} x^\gamma,\qquad
	\int_{x}^{\infty} e^{-\mu y} y^\gamma dy \leq C\cdot e^{-\mu x} x^\gamma.
\end{equation}
We find
\begin{equation}
	\begin{split}
	  \int_{a}^{\infty} x^\alpha\cdot  |G_z(x,y)|\cdot y^\beta {\bf dy} 
        &\leq C \cdot e^{-\mu x} x^{\alpha-1} 
          	\int_{a}^{x} e^{\mu y} 	y^{\beta-1} dy \\
	&\qquad + C \cdot e^{\mu x} x^{\alpha-1} \int_{x}^{\infty} e^{-\mu y} 
	y^{\beta-1} dy \\
	&\leq C\cdot x^{\alpha+\beta -2}\leq C ,
	\end{split}
\end{equation}
and reversing the roles of $\alpha$, $\beta$,
\begin{equation}
	\int_{a}^{\infty} x^\alpha\cdot |G_z(x,y)| \cdot y^\beta {\bf dx} \leq C .
\end{equation}
With these inequalities the claim follows from
Schur's test.
\end{proof}

\begin{prop}\label{P.regdet.2}
Fix $\mu$ and let $-z^2$ be in the resolvent set of $D_\mu$.
For $\delta>0$ the operator $X^{\frac{1}{2}-\delta}\cdot (D_\mu +z^2)^{-1/2}$
is of Hilbert-Schmidt class resp. $X^{\frac{1}{2}-\delta}\cdot (D_\mu
+z^2)^{-1}\cdot X^{\frac 12 -\delta}$ is of trace class. 

Moreover, for real numbers $\alpha$, $\beta$ with $\alpha+\beta<\frac{3}{2}$,
the operator $X^{\alpha}\cdot (D_\mu + z^2)^{-1}\cdot X^\beta$ is a Hilbert-Schmidt operator.
\end{prop}

\begin{proof} Since for any two $z_1$, $z_2$ with $-z_1^2$, $-z_2^2$ in the
resolvent set the operator $	(D_\mu +z_1^2)^{-1}\cdot	(D_\mu +z_2^2) $ is bounded,
it suffices to prove the claim for $z\geq 0$. Then
\begin{equation}
	\begin{split}
    \|X^{\frac{1}{2}-\delta}\cdot (D_\mu + z^2)^{-\frac{1}{2}}\|_{\textrm{HS}} 
	&=\Tr\bl X^{\frac{1}{2}-\delta}\cdot (D_\mu +z^2)^{-1}\cdot X^{\frac{1}{2}-\delta}\br\\
	&=\int_{a}^{\infty} x^{1-2\delta}\cdot G_z(x,x) dx \leq C\int_{a}^{\infty} x^{-1-2\delta} dx < \infty,
	\end{split}
\end{equation}
since $\delta>0$. Here we have used that $G_z(x,x) = O(x^{-2}) $ as
$x\to\infty$ by \Eqref{a.e.I.x.inf} and 
\eqref{a.e.K.x.inf}.

For the second part, pick $\delta>0$ such that $2(\alpha+\beta)+2\delta<3$.
Let $k(x,y)$ be the Schwartz kernel of $(D_\mu+z^2)\ii X^{2\ga}
(D_\mu+z^2)\ii$. From the proof of the first part we infer that 
$X^{\frac 12 -\delta}\cdot (D_\mu+z^2)\ii\cdot X^{\frac 12-\delta}$ is trace class
and from Lemma \ref{L.regdet.1} we infer that 
$X^{2\beta+\delta-\frac{1}{2}}\cdot (D_\mu + z^2)\ii \cdot X^{2\alpha+\delta-\frac{1}{2}}$
is bounded. Consequently, the product of these two operators, 
$ X^{2\beta+\delta-\frac{1}{2}}\cdot(D_\mu +z^2)\ii \cdot X^{2\alpha} \cdot 
	(D_\mu + z^2)\ii\cdot X^{\frac 12-\delta} $
is trace class and Mercer's Theorem implies that
\begin{equation}
\begin{split}
    \Tr\bl  X^{2\beta+\delta-\frac{1}{2}}\cdot (D_\mu +z^2)\ii \cdot X^{2\alpha} \cdot 
	        &   (D_\mu + z^2)\ii\cdot X^{\frac 12-\delta} \br \\
            & = \int_a^\infty x^{2\gb+\delta-\frac 12} \cdot k(x,x)\cdot
                 x^{\frac 12 - \delta} dx \\
            & = \int_a^\infty x^\gb \cdot k(x,x) \cdot x^\gb dx.
\end{split}
\end{equation}             
On the other hand the operator $X^\gb \cdot (D_\mu+z^2)\ii\cdot X^{2\ga}\cdot
(D_\mu +z^2)\ii\cdot X^\gb $ is non-negative. Hence from Mercer's Theorem in the version of
Reed and Simon \cite[\S XI.4, Lemma on p. 65]{Reed} we infer
that indeed
\begin{equation}
\begin{split}
             \int_a^\infty x^\gb \cdot k(x,x) \cdot x^\gb dx
                & = \Tr\bl X^\gb \cdot (D_\mu+z^2)\ii \cdot X^{2\ga}\cdot
                (D_\mu+z^2)\ii\cdot X^\gb \br \\
                & = 
	     \|X^{\alpha}\cdot (D_\mu+z^2)^{-1}\cdot X^{\beta}\|_{\rm HS}^2.
\end{split}
\end{equation}
Since we know that the left hand side is finite we reach the conclusion.
\end{proof}
The argument using the kernel and applying Mercer's Theorem twice was
necessary to justify the manipulation
\begin{equation}
\begin{split}
    \Tr\bl  X^{2\beta+\delta-\frac{1}{2}}\cdot (D_\mu +z^2)\ii \cdot &X^{2\alpha} \cdot 
	           (D_\mu + z^2)\ii\cdot X^{\frac 12-\delta} \br  \\
       &= \Tr\bl  X^{\gb}\cdot (D_\mu +z^2)\ii \cdot X^{2\alpha} \cdot 
	           (D_\mu + z^2)\ii\cdot X^{\gb} \br.
\end{split}
\end{equation}

This rearrangement is not trivial since a priori 
$X^{\gb}\cdot (D_\mu +z^2)\ii \cdot X^{2\alpha} \cdot (D_\mu + z^2)\ii\cdot X^{\gb}$
need not be trace class.

\begin{lemma}\label{lemma2}
	For $W\in L^1[a,\infty)$ and $\eps\geq 0$ we have
	\begin{equation}
	\|X^{1-\frac{\eps}{2}} |W|^\frac{1}{2} (D_\mu+z^2)^{-\frac{1}{2}} 
	\|_{\rm HS} = O(z^{-\frac{\min(1,\eps)}{2}}),\quad z\to\infty.
	\end{equation}
\end{lemma}

\begin{proof} 
	For $\eps> 1$ we estimate,
	\begin{equation}
	\|X^{1-\frac{\eps}{2}} |W|^{\frac{1}{2}}(D_\mu +z^2)^{-\frac{1}{2}}\|_{\rm HS} 
	\leq \||W|^{\frac{1}{2}}(D_\mu +z^2)^{-\frac{1}{2}}\|_{\rm HS},
	\end{equation} hence it suffices to prove the Lemma for $0\leq \eps\leq 1$.
	
	The same argument as in the proof of Proposition \ref{P.regdet.2} shows that
	\begin{equation}
	X^{1-\frac{\eps}{2}} |W|^\frac{1}{2} (D_\mu+z^2)^{-1} 
	|W|^{\frac{1}{2}}X^{1-\frac{\eps}{2}}
	\end{equation} 
	is trace class, hence 
	$X^{1-\frac{\eps}{2}}|W|^{\frac{1}{2}}(D_\mu+z^2)^{-\frac{1}{2}}$ is 
	a Hilbert-Schmidt operator. We have
	\begin{equation}
	\begin{split}
	\|X^{1-\frac{\eps}{2}} |W|^\frac{1}{2} (D_\mu+z^2)^{-\frac{1}{2}}\|_{\rm HS}^2 &= 
	\Tr\left(X^{1-\frac{\eps}{2}} |W|^\frac{1}{2} (D_\mu+z^2)^{-1} 
	|W|^{\frac{1}{2}}X^{1-\frac{\eps}{2}}\right)\\
	&=\int_{a}^{\infty} x^{2-\eps} |W(x)|G_z(x,x) dx\\
	&=\frac{z}{\mu}\int_{\frac{\mu a}{z}}^{\infty} \left(\frac{xz}{\mu}\right)^{2-\eps} 
	\left|W\left(\frac{xz}{\mu}\right)\right| 
	G_z\left(\frac{xz}{\mu},\frac{xz}{\mu}\right)dx.
	\end{split}
	\end{equation}
	Before we estimate the integral note that, for all $x\geq 0$,
	\begin{equation}
	\frac{x^{1-\eps}}{(1+x^2)^{\frac{1}{2}}}\leq 1.
	\end{equation} 
	According to \Eqref{resolvent-kernel} we split the integral into a sum and estimate each 
	summand separately. For the first integral, we use the asymptotic expansion 
	\Eqref{uniform1},
	\begin{equation}
	\begin{split}
	\frac{z^{2-\eps}}{\mu^{2-\eps}}\int_{\frac{\mu a}{z}}^{\infty}x^{1-\eps} 
	\left|W \left(\frac{xz}{\mu}\right)\right|I_z(xz)K_z(xz)dx &\leq C_1 
	\frac{z^{1-\eps}}{\mu^{2-\eps}}\int_{\frac{\mu a}{z}}^{\infty}
	\frac{x^{1-\eps}}{(1+x^2)^{\frac{1}{2}}} 
	\left|W \left(\frac{xz}{\mu}\right)\right| dx\\
	&\leq C_2 z^{-\eps}.
	\end{split}
	\end{equation}
	For the second integral we use the asymptotic expansions  
	\Eqref{asymp.nu.I.xfixed}, \eqref{asymp.nu.K.xfixed} and \eqref{a.e.K.nu.inf}, and 
	obtain
\begin{equation}
	\frac{z^{2-\eps}}{\mu^{2-\eps}}\frac{I_z(\mu)}{K_z(\mu)} 
	\int_{\frac{\mu a}{z}}^{\infty}x^{1-\eps}\left|W 
	\left(\frac{xz}{\mu}\right)\right| K_z^2(xz) dx \leq C_3 
	\frac{z^{1-\eps}}{\mu^{2-\eps}}
	\int_{\frac{\mu a}{z}}^{\infty} \left|W \left(\frac{xz}{\mu}\right)\right| dx
	\leq C_4 z^{-\eps}.\qedhere
	\end{equation}
\end{proof}

\begin{theorem}\label{Theo.Pert.Trace} 
Let $W\in L^1[a,\infty)$ and $\eps>0$. Let $R_a$ be a
boundary condition (Dirichlet or generalized Neumann)
at $a$. By slight abuse of notation let $D_\mu:= D_\mu(R_a)$.
Then the resolvent, $(D_\mu +X^{2-\eps} W + z^2)^{-1}$, is trace class
and
\begin{equation}\label{eq.regdet.3}
\|(D_\mu +X^{2-\eps} W + z^2)^{-1}  -(D_\mu + z^2)^{-1}\|_{\rm tr}
   = O(z^{-2-\min(1,\eps)}),\qquad \text{ as }	z\to \infty.
\end{equation}
Consequently,
\begin{equation}\label{eq.regdet.4}
\begin{split}
    \Tr(D_\mu &+ X^{2-\eps} W + z^2)^{-1} 
      = \Tr (D_\mu +z^2)^{-1} + O(z^{-2-\min(1,\eps)}) \\
     & = b_0 \cdot z\ii \cdot \log z + a_0\cdot z\ii + a_1 \cdot z^{-2}
             + O(z^{-2-\min(1,\eps)}), \quad\text{ as } z\to\infty,
\end{split}
\end{equation}
where the constants $b_0, a_0, a_1$ are those of Prop. \ref{Asympt.Dt}.
Consequently, the zeta-regularized determinant of the operator
$D_\mu + X^{2-\eps} W + z^2$ is well-defined as explained in 
the Introduction Sec.~\ref{s.intro} for any $z \geq 0$.
\end{theorem}

\begin{proof} We apply the Neumann series, a priori formally,
\begin{equation}
	\begin{split}
	&(D_\mu + X^{2-\eps}W + z^2)^{-1}  - (D_\mu  + z^2)^{-1}  \\
	&=\sum_{n=1}^{\infty} 
	(-1)^n (D_\mu +z^2)^{-\frac{1}{2}} 
	\left[(D_\mu+z^2)^{-\frac{1}{2}} 
	X^{1-\frac{\eps}{2}}W X^{1-\frac{\eps}{2}} (D_\mu + z^2)^{-\frac{1}{2}} 
	\right]^n 
	(D_\mu + z^2)^{-\frac{1}{2}}\\
	&=\sum_{n=1}^{\infty} (-1)^n E_n(z).
	\end{split}
\end{equation}
We estimate each summand using Lemma \ref{lemma2},
\begin{equation}
	\begin{split}
	\|E_n&(z)\|_{\rm tr} \leq \|(D_\mu+z^2)^{-\frac{1}{2}}\|_{L^2}^2
	\|(D_\mu+z)^{-\frac{1}{2}} X^{1-\frac{\eps}{2}}|W| 
	X^{1-\frac{\eps}{2}} (D_\mu + z^2)^{-\frac{1}{2}} \|^n_{\rm tr}\\
	&\leq \|(D_\mu+z^2)^{-1}\|_{L^2} \| (D_\mu+z^2)^{-\frac{1}{2}} 
	X^{1-\frac{\eps}{2}}|W|^{\frac{1}{2}} 
	\|_{\rm HS}^n \||W|^{\frac{1}{2}} X^{1-\frac{\eps}{2}} (D_\mu + 
	z^2)^{-\frac{1}{2}}\|_{\rm HS}^n\\
	&\leq C z^{-2} \| (D_\mu+z^2)^{-\frac{1}{2}} 
	X^{1-\frac{\eps}{2}}|W|^{\frac{1}{2}} 
	\|_{\rm HS}^n \||W|^{\frac{1}{2}} X^{1-\frac{\eps}{2}} (D_\mu + 
	z^2)^{-\frac{1}{2}}\|_{\rm HS}^n\\
  &\leq C \cdot z^{-2} \cdot \bl z^{-\min(1,\eps)}\br^n.
	\end{split}  
\end{equation}
This shows that for $z$ large enough the Neumann series indeed
converges in the trace norm and that \Eqref{eq.regdet.3}
holds.
\end{proof}

\section{Variation formula and the determinant of the perturbed operator}
\label{s.regdet}

In this section we prove our main result, Theorem \ref{T.main}. The next theorem ge\-ne\-ra\-li\-zes
Theorem \ref{Theo.Vert}. 

\begin{theorem}\label{W-variation-thm}
Let $W_t$ be a differentiable family of functions in $L^{1}[a,\infty)$
and	$\eps>0$. As a notational convenience we take $t_0 = 0$. Fix $\nu \geq 0$ and 
consider the perturbed operator
\begin{equation}
	H_t:= D_\mu + X^{2-\eps}W_t.
\end{equation} 
Furthermore, let $R_a$ be a boundary condition at $a$. Let $\psi_t$, 
$\phi_t$ be
a fundamental system of solutions to the equation $(H_t + \nu^2)u = 0$, where $\phi_t$ is
normalized in the sense of \Eqref{INTRO.eq.3} and $\psi_t$ satisfies
\begin{equation}\label{norm.sec7.infinity}
\lim\limits_{x\to \infty} \psi_t(x) \, \sqrt{x}  \, K_\nu(\mu x)^{-1} = 1.
\end{equation}
Assume that $H_{0}$ is invertible. Then we have for the variation of the 
zeta-regularized
determinant
\begin{equation}
	\frac{d}{dt} \log\detz (H_t + \nu^2) \bigr|_{t=0} = \frac{d}{dt} \log\bl x^2
  \cdot W(\psi_t,\phi_t)(x)\br \bigr|_{t=0}.
	\end{equation}
\end{theorem}
This theorem contains Theorem \ref{Theo.Vert} as a special case;
just put $W_t = t^2 X^{-3/2},$ $t = z,$ $\nu = 0,$ and $\eps = 1/2$. 

\begin{proof}
This theorem is a consequence of Proposition \ref{GEN-Trace-Formula}, Theorem 
\ref{T.Wronsk}, Corollary \ref{Wron.Corol.3}, Lemma 
\ref{lemma2} and Theorem 
\ref{Theo.Pert.Trace}. 
Firstly, by Theorem \ref{Theo.Pert.Trace} 
the zeta-determinant is defined for $H_t+\nu^2$ for any $t$ and $\nu \geq 0$,
and the difference,
\begin{equation}
	\log\detz (H_t + \nu^2) - \log\detz (H_{0} + \nu^2) = -2\int_{\nu}^{\infty} 
	z\Bl \Tr(H_t +z^2)^{-1} - \Tr(H_{0} +z^2)^{-1}\Br dz,
\end{equation} 
is well-defined by \Eqref{eq.regdet.3}. 
Moreover, Theorem \ref{Theo.Pert.Trace} shows that the map $t\mapsto (H_t+z^2)\ii$ 
is differentiable as a map into the space of trace class operators
and hence
\begin{equation}
	\begin{split}
    	\frac{d}{dt} z \Tr(H_t + z^2)\ii \bigr|_{t=0}
        &= z \Tr\bl (H_t + z^2)\ii (\pl_t W_t) (H_t+z^2)\ii \br\bigr|_{t=0}\\
	      &= - \frac{1}{2} \frac{d}{dz}\Tr\Bl (\pl_t W_t) (H_t + z^2)\ii \Br\bigr|_{t=0},
\end{split}
\end{equation} 
and it follows from the proof of Lemma \ref{lemma2} that the latter is $O(z^{-2-\min(1,\eps)})$ 
as $t\to\infty$ locally uniformly
in $t$. Therefore, by Dominated Convergence Theorem, we may differentiate under
the integral and find
\begin{equation}
  \frac{d}{dt} \log \detz \bl H_t + \nu^2 \br \bigr|_{t=0}= \Tr((\pl_t W_{t})\bl H_{t} 
  + \nu^2 \br\ii) 
  \bigr|_{t=0}.
	\end{equation} 
By Theorem \ref{T.Wronsk}, there exists a fundamental system of solutions to 
$(H_t+\nu^2)u=0$ given by $h_{1,t}$ and $h_{2,t}$, such that $h_{1,t}$ satisfy 
\Eqref{norm.sec7.infinity}. Consider $\phi_t$ a linear combination of $h_{1,t}$ and 
$h_{2,t}$ normalized as \Eqref{INTRO.eq.3} and $\psi_t = h_{1,t}$.
Now the result follows from Corollary \ref{Wron.Corol.3} and Proposition 
\ref{GEN-Trace-Formula}.
\end{proof}

\subsubsection*{Proof of Theorem \ref{T.main}}
We conclude the section with a proof of our main result, Theorem \ref{T.main}. 
We conclude by Theorem \ref{W-variation-thm} for any $\gamma < 2$,  
potential $V\in X^{\gamma}L^1[a,\infty)$ and boundary conditions $R_a$ at $x=a$
\begin{align}
\detz \bl H(R_a) + \nu^2 \br = c_0(a, \mu) \cdot a^2 \cdot W(\psi, \phi)(a),
\end{align}
where the constant $c_0(a,\mu)$ does not depend on $V$. The same argument as in Theorem 
\ref{T.ModelDetFormula} shows that $c_0(a, \mu)$ does not depend on the boundary 
condition. In particular, the equality holds in the special case of a trivial 
potential $V\equiv 0$ and hence by Theorem \ref{T.ModelDetFormula}
\begin{align}
c_0(a, \mu) &= \frac{\detz \left( H(R_a) + \nu^2\right)}{a^2\cdot W(\psi, \phi)(a)} = 
\sqrt{\frac{2}{\pi}}.
\end{align}
This completes the proof\footnote{A posteriori $c(a,\mu)$ does not depend on $a$ 
	and 
	$\mu$ either.}.\qed


\bibliography{mlbib,localbib}
\bibliographystyle{amsalpha-lmp}

\end{document}